\documentclass[11pt]{article}
\usepackage{amsmath,amsthm,amssymb}
\usepackage[T1]{fontenc}
\usepackage[utf8]{inputenc}
\usepackage[dvipdf]{graphicx}
\usepackage{color}
\usepackage{epstopdf}
\usepackage{dsfont}
\usepackage{mathtools}
\usepackage{enumerate}
\usepackage{mathrsfs}
\usepackage[normalem]{ulem}
\usepackage[colorlinks=true,citecolor=red,linkcolor=db,urlcolor=blue,pdfstartview=FitH]{hyperref}
\usepackage[numbers]{natbib}

\usepackage{xcolor}

\definecolor{db}{RGB}{0, 0, 130}
\definecolor{rp}{rgb}{0.25, 0, 0.75}
\definecolor{dg}{rgb}{0, 0.6, 0}

\numberwithin{equation}{section}

\newtheorem{thm}{Theorem}[section]

\newtheorem{definition}{Definition}[section]
\newtheorem{prop}{Proposition}[section]

\newtheorem{example}{Example}[section]
\newtheorem{assumption}{Assumption}[section]
\newtheorem{lemma}{Lemma}[section]

\newtheorem{remark}{Remark}[section]

\def\pa{\partial}

\def\Hc{\mathcal H}

\def\Pc{\mathcal P}

\def\N{\mathbb N}

\def\R{\mathbb R}
\def\S{\mathbb S}
\def\T{\mathbb T}
\def\Z{\mathbb Z}

\def\ub{\bar u}
\def\Vb{\overline{V}}
\def\Jb{\bar{J}}

\def\x{\times}

\def\eps{\varepsilon}

\title{Quantitative Particle Approximation for Controlled Nonlinear Filtering
}
\author{Erhan Bayraktar
        \footnote{Department of Mathematics, University of Michigan. erhan@umich.edu. E. Bayraktar is partially supported by the NSF Grants DMS-2507940 and DMS-2406232 and by the Susan M. Smith Professorship.}
        \and Ibrahim Ekren
        \footnote{Department of Mathematics, University of Michigan. iekren@umich.edu. I. Ekren is partially supported by the NSF Grant DMS-2406240.}
        \and Xihao He
        \footnote{Department of Mathematics, University of Southern California.
        xihaohe@usc.edu}
        \and Xin Zhang
        \footnote{Department of Finance and Risk Engineering, New York University. xz1662@nyu.edu. X. Zhang is partially supported by the NSF Grant DMS-2508556.}
        }
\date{\today}

\begin{document}

\maketitle

\begin{abstract}
We estimate convergence rates of value functions for particle approximations
of a controlled nonlinear filtering problem. The state is a McKean--Vlasov
diffusion on the flat torus, driven by hidden idiosyncratic noise and observed
common noise. The filter---the conditional law of the state given the
observations---serves as the state variable of the control problem, and the
associated value function solves a second-order Hamilton--Jacobi--Bellman
equation on the Wasserstein space. We approximate this problem by a
centralized \(N\)-particle control problem with independent idiosyncratic
noises and a common observation noise. The framework accommodates nonseparable
rewards and controlled drifts. Since a single control is applied to the entire
population, the Hamiltonian is defined by an optimization performed after
integration over the population. Under smoothness of the data, uniform
ellipticity, and regularity of this Hamiltonian, we establish uniform
value-function error bounds of order \(N^{-1/6}\) for \(d=1\),
\(N^{-1/6}(\log N)^{1/3}\) for \(d=2\), and \(N^{-1/(3d)}\) for \(d>2\).
The proof combines a translation lift in the common-noise direction,
Fourier--Wasserstein inf- and sup-convolutions, viscosity comparison, and
particle derivative estimates uniform in \(N\).
\end{abstract}


\section{Introduction}
\label{sec:introduction}

Partially observed stochastic control problems are often made Markovian by
passing from the unobserved state to its conditional law given the observations.
This conditional law is the filter. It serves as the information state, and
the associated value function is defined on a space of probability measures.
For background on partially observed control and
filtering, see, for instance, \cite{BensoussanPartialObservation}; for recent
dynamic programming and Wasserstein HJB formulations based on randomized
filtering, see
\cite{BandiniCossoFuhrmanPham2018,BandiniCossoFuhrmanPham2019}. The present
paper studies such a controlled filtering problem in which the observation
noise is a common noise. The filter is therefore a measure-valued controlled
process, and the corresponding dynamic programming equation is a second-order
HJB equation on the Wasserstein space. We also draw on the dynamic programming theory for
McKean--Vlasov control with open-loop controls
\cite{BayraktarCossoPham2018,PhamWeiMkvDPP,DjetePossamaiTanDPP}.

Particle methods are fundamental in nonlinear filtering. Classical
sequential Monte Carlo and interacting-particle approximations approximate the
filter itself, or the Zakai and Kushner--Stratonovich equations; see, among
many others,
\cite{GordonSalmondSmith1993,Kitagawa1996MonteCarloFilter,
DoucetDeFreitasGordon2001,CrisanLyonsMeasureValued,
CrisanGainesLyonsZakai,CrisanLyonsKushnerStratonovitch,
CrisanDoucetParticleFiltering,DelMoral2004FeynmanKac}. Those works address a
different question from the one considered here: we approximate the
\emph{control} problem written in terms of the filter and estimate the error at
the level of the optimal value function. In this direction, finite-population approximation and
propagation-of-chaos results for McKean--Vlasov control provide the relevant
control-theoretic background
\cite{McKeanPropagation,SznitmanPropagation,LackerControlledMKV,djete,
DjetePossamaiTanDPP,FournierGuillin,GangboMayorgaSwiech,GermainPhamWarin,
CardaliaguetSouganidis,CardaliaguetDaudinJacksonSouganidis,
DaudinDelarueJackson,CardaliaguetJacksonMimikosSouganidis}.
To our knowledge, quantitative value-function rates for particle
approximations of controlled nonlinear filtering problems, in the
common-control Hamiltonian setting treated below, have not previously been
available.

Our starting point is the finite-population approximation developed by
Bouchard and Tan \cite{BouchardTan} for mean-field control problems with
controls adapted to the common-noise filtration. In their formulation, the
central planner in the \(N\)-particle system observes all particle states and
chooses a common signal. This produces a time-consistent finite-population
problem approximating the common-noise-adapted mean-field control problem, and
in a Markovian regime with affine drift, constant volatilities, and a separable
running reward, they obtain the optimal weak value-function rate \(N^{-1}\).
Their work also shows how this common-noise-adapted approximation applies to
classical partially observed control problems.

This paper establishes a value-function convergence rate for a more flexible
controlled nonlinear filtering problem. We work on the flat torus and
allow controlled drifts and nonseparable running rewards. In particular,
unlike the special setting of \cite{BouchardTan}, the running reward need not
decompose as
\(
\ell(t,\mu,a)=\ell_0(t,a)+F(\mu),
\)
where \(\mu\) denotes the state distribution and \(a\) the control. Our rate is
weaker than the optimal \(N^{-1}\) weak rate available under the affine-drift,
separable-reward assumptions of \cite{BouchardTan}, but the proof does not rely
on the explicit Gaussian-convolution representation that reduces the value
function to a finite-dimensional HJB equation in that special case. Instead, we
work directly with the Wasserstein equation and its \(N\)-particle
approximation.

The Hamiltonian is a central structural feature. Because the same action is
used for the whole particle system, the maximization is performed after
averaging over the particles. To express the limiting Hamiltonian, let
\(\eta\in\mathcal P_2(\mathbb T^d\times\mathbb R^d)\) be a probability measure
on pairs \((x,p)\), where \(x\) is the state variable and \(p\) is the
corresponding gradient variable in the HJB equation. Define the state marginal
of \(\eta\) by
\(
m_\eta(B):=\eta(B\times\mathbb R^d),
\, B\subseteq\mathbb T^d\ \text{Borel}.
\)
Then the relevant limiting Hamiltonian is
\begin{align}\label{eq:hamiltonian}
\mathfrak H(t,\eta)
=
\sup_{a\in A}
\int_{\mathbb T^d\times\mathbb R^d}
\left[
\ell(t,x,m_\eta,a)+b(t,x,m_\eta,a)\cdot p
\right]\eta(dx,dp).
\end{align}
For the \(N\)-particle problem, \(\eta\) is the empirical measure
\(N^{-1}\sum_{i=1}^N\delta_{(x_i,p_i)}\), and consequently
\(m_\eta=N^{-1}\sum_{i=1}^N\delta_{x_i}\). The supremum is outside the integral
because a single action must be used for every particle. This differs from the
pointwise Hamiltonian obtained by optimizing the action separately for each
particle and then averaging. Related common-control and intrinsic
viscosity-solution structures appear in
\cite{BayraktarCecchinChakrabortyRegimeSwitching,BurzoniIgnazioReppenSoner,SonerYanTorus}.

Our proof follows the PDE strategy used in quantitative McKean--Vlasov control
and finite-dimensional approximation results, including the adjoint and
particle-derivative estimates developed in work of Cardaliaguet, Daudin,
Jackson, Souganidis, and collaborators, and in our earlier work on
particle convergence rates and comparison for second-order Wasserstein PDEs;
see
\cite{CardaliaguetDaudinJacksonSouganidis,DaudinDelarueJackson,
CardaliaguetJacksonMimikosSouganidis,BayraktarEkrenZhangRate,
BayraktarEkrenZhangComparison,BayraktarEkrenHeZhangComparison,
BayraktarEkrenHeZhangCommonNoise}. In our framework, however, the optimization defining the Hamiltonian is
performed after integration over the population, and consequently the
particle-derivative estimates developed in these references cannot be applied
directly. We show that they nevertheless remain valid in our setting. The common noise
creates a second-order derivative in the measure variable. For constant common
noise, we remove this term by translating the measure and introducing an
additional torus variable. We then compare the translated limiting equation
with the translated \(N\)-particle equation by means of
Fourier--Wasserstein inf- and sup-convolutions. The regularity of the optimized
common-control Hamiltonian gives precisely the mean-field scaling needed to
differentiate the \(N\)-particle equation uniformly in \(N\).

The result yields a particle approximation of the value function for a
controlled nonlinear filtering problem, while preserving the Hamiltonian
structure of common-signal McKean--Vlasov control. It
complements the qualitative and optimal-rate results of \cite{BouchardTan}: we
obtain an explicit value-function rate under more general data, without
attempting to recover the optimal weak rate in the special structures where a
Taylor expansion of a smooth finite-dimensional reduction is available.

The rest of the paper is organized as follows.
Subsection~\ref{subsec:notation} fixes notation and the basic metric setting.
Section~\ref{sec:formulation} formulates the controlled filtering problem and
its centralized \(N\)-particle approximation.
Section~\ref{sec:main-result} states the main rate theorem and discusses the
Hamiltonian regularity assumption. Section~\ref{sec:proofs} proves the main
rate theorem. 

\subsection{Notation}
\label{subsec:notation}

Let $\N_+$ denote the set of strictly positive integers, and let $T$ be a
finite horizon with $0 < T < +\infty$.
We fix $d \in \N_+$ throughout the paper and work on the torus
$\T^d := \R^d/(2\pi\Z)^d$. 

For $N \in \N_+$, write $[N] := \{1,\ldots,N\}$.
For any Polish space $(E,d)$, let $\Pc(E)$ denote the set of all probability
measures on $E$.
For any $\rho \in \mathcal{P}(\mathbb T^d)$ and any bounded continuous
function $f:\T^d \to \mathbb R$, we write
$\langle f,\rho\rangle := \int f(x) \, \rho(dx)$.

Fix an integer \(k^\ast>d/2+2\). The Fourier--Wasserstein $2$-distance
$\rho_F$ between two probability measures
$\rho_1,\rho_2 \in \Pc(\T^d)$ is defined by
    \begin{equation*}
        \rho_F^2(\rho_1,\rho_2)
        ~ := ~
        \sum_{k \in \Z^d}\frac{|F_k(\rho_1-\rho_2)|^2}{(1 + |k|^2)^{k^\ast}},
    \end{equation*}
where, for \(k\in\Z^d\),
$f_k(x) := (2\pi)^{-\frac{d}{2}}e^{i k \cdot x}$ and
$F_k(\rho) := \langle f_k, \rho \rangle$. For any complex number $z$, we
denote its conjugate by $z^*$.

For $\theta_1=(t_1,\mu_1,z_1)$ and
$\theta_2=(t_2,\mu_2,z_2)$ in
\(\Theta:=[0,T] \x \Pc(\T^d) \x \T^d\), define
    \[
        d_F^2(\theta_1,\theta_2)=|t_1-t_2|^2+|z_1-z_2|^2+\rho_F^2(\mu_1,\mu_2).
    \]
We denote by $\S_d$ the space of all real $d \x d$ matrices, equipped with
the Frobenius norm $|\cdot|$.

Set \(E:=\mathbb T^d\times\mathbb R^d\), and let
\(\mathcal P_2(E)\) denote the set of probability measures on \(E\) with
finite second moment. Let
\(\pi_x:E\longrightarrow\mathbb T^d\), \(\pi_x(x,p):=x\), be the projection
onto the state variable. For \(\eta\in\mathcal P_2(E)\), we write
\(m_\eta:=(\pi_x)_\sharp\eta\in\mathcal P(\mathbb T^d)\) for its spatial
marginal.

\section{Problem Formulation}
\label{sec:formulation}
We start from the controlled filtering problem and then introduce the
centralized finite-particle approximation. The value function in
Subsection~\ref{subsec:filtering-problem} is the object we want to compute.
Subsection~\ref{subsec:finite-particle-problem} defines the finite-dimensional
centralized control problem that approximates it; the main result,
Theorem~\ref{thm:rate}, quantifies this approximation.

\subsection{Controlled filtering problem on the torus}
\label{subsec:filtering-problem}

Let
$q:\mathbb R^d\longrightarrow \mathbb T^d$
denote the canonical quotient map. Let
\((\Omega,\mathcal F,\mathbb F,\mathbb P)\) be a complete filtered
probability space satisfying the usual conditions and supporting an
\(\mathbb R^m\)-valued Brownian motion \(W\) and an
\(\mathbb R^{m_0}\)-valued Brownian motion \(B\), which are independent.
Let the control space \((A,\rho)\) be a Polish space and fix \(a_0\in A\).

Fix an initial time \(t\in[0,T]\) and an initial belief
$\mu\in\mathcal P(\mathbb T^d).$
We assume that the probability space
is sufficiently rich to support an \(\mathcal F_t\)-measurable
\(\mathbb T^d\)-valued random variable \(\xi^{t,\mu}\) such that
$\mathcal L^{\mathbb P}(\xi^{t,\mu})=\mu,$
and such that \(\xi^{t,\mu}\) is independent of the future increments
$\bigl(W_s-W_t,B_s-B_t\bigr)_{t\leq s\leq T}.$

The process \(B\) represents the observable noise, whereas the torus-valued
state and the idiosyncratic noise \(W\) are not observed.
Starting from time \(t\), define the observation filtration
\[
\mathcal F_s^t
:=
\sigma\bigl(B_r-B_t:t\leq r\leq s\bigr),
\qquad s\in[t,T],
\]
and let \(\mathbb F^t=(\mathcal F_s^t)_{t\leq s\leq T}\). The information
available before time \(t\) is summarized by the initial belief \(\mu\),
while \(\mathbb F^t\) records the observations received after time \(t\).

The set of admissible controls starting from time \(t\) is
\begin{equation*}
\mathcal A_t
:=
\left\{
\alpha:[t,T]\times\Omega\longrightarrow A:
\alpha \text{ is }\mathbb F^t\text{-predictable},
\displaystyle
\mathbb E\left[
\int_t^T\rho(\alpha_s,a_0)^2\,ds
\right]<\infty
\right\}.
\end{equation*}

Let the drift coefficient be
\(b:[0,T]\times\mathbb T^d\times\mathcal P(\mathbb T^d)\times A
\longrightarrow\mathbb R^d\), and let
$\sigma\in\mathbb R^{d\times m}$ and
$\sigma_0\in\mathbb R^{d\times m_0}$ be constants.
The lift of \(b\) to \(\mathbb R^d\) is \(\mathbb Z^d\)-periodic in the
state variable. For each \(\alpha\in\mathcal A_t\), the controlled
torus-valued state
process \(X^{t,\mu,\alpha}\) is defined as follows. Choose an
\(\mathbb R^d\)-valued measurable lift \(\widehat \xi^{t,\mu}\) of
\(\xi^{t,\mu}\), namely
\(
q(\widehat \xi^{t,\mu})=\xi^{t,\mu}.
\)
Let \(\widehat X^{t,\mu,\alpha}\) solve the lifted SDE
\begin{equation*}
\begin{aligned}
\widehat X_s^{t,\mu,\alpha}
={}&
\widehat \xi^{t,\mu}
+\int_t^s
b\bigl(r,q(\widehat X_r^{t,\mu,\alpha}),
       \pi_r^{t,\mu,\alpha},\alpha_r\bigr)\,dr
+\int_t^s \sigma \,dW_r+\int_t^s
\sigma_0\,dB_r,
\qquad s\in[t,T].
\end{aligned}
\end{equation*}
The torus-valued state is then
\(
X_s^{t,\mu,\alpha}
:=
q\bigl(\widehat X_s^{t,\mu,\alpha}\bigr),
\, s\in[t,T].
\)
Because the coefficients are periodic in the state variable, the process
\(X^{t,\mu,\alpha}\) does not depend on the particular lift
\(\widehat \xi^{t,\mu}\).

The conditional distribution, or filter, of the unobserved torus-valued
state given the observations is
\begin{equation*}
\pi_s^{t,\mu,\alpha}
:=
\mathcal L^{\mathbb P}
\bigl(X_s^{t,\mu,\alpha}\mid\mathcal F_s^t\bigr)
\in\mathcal P(\mathbb T^d),
\qquad s\in[t,T].
\end{equation*}
In particular,
$
\pi_t^{t,\mu,\alpha}=\mu,
\, \,  \mathbb P\text{-a.s.}
$

Let the running and terminal reward functions be
\[
\ell:[0,T]\times\mathbb T^d\times \Pc(\T^d) \x A\longrightarrow\mathbb R
\quad\text{and}\quad
G:\Pc(\mathbb T^d)\longrightarrow\mathbb R.
\]
The value function of the control problem started from \((t,\mu)\) is
therefore
\begin{equation*}
V(t,\mu)
:=
\sup_{\alpha\in\mathcal A_t}J(t,\mu,\alpha), \quad J(t,\mu,\alpha)
:=
\mathbb E\left[
\int_t^T
\langle \ell\bigl(s,\cdot, \pi_s^{t,\mu,\alpha},\alpha_s\bigr),\pi_s^{t,\mu,\alpha}\rangle\,ds
+
G(\pi_T^{t,\mu,\alpha})
\right].
\end{equation*}
In particular, the terminal condition is
\(
V(T,\mu)=G(\mu),
\, \mu\in\mathcal P(\mathbb T^d).
\)

\subsection{Centralized finite-particle control problem}
\label{subsec:finite-particle-problem}

Fix \(N\geq 1\). Let
\((\Omega^{N},\mathcal F^{N},\mathbb P^{N})\) be a complete probability space
supporting an \(\mathbb R^{m_0}\)-valued Brownian motion \(B\), and
independent \(\mathbb R^m\)-valued Brownian motions
\((W^k)_{k=1}^{N}\). The Brownian motion \(B\) represents the common noise,
while \(W^k\) is the idiosyncratic noise of the \(k\)-th particle.

For \(t\in[0,T]\), we denote by
\(\mathbb F^{N,t}=(\mathcal F^{N,t}_{s})_{t\leq s\leq T}\) the filtration
generated by the increments of these processes after time \(t\), namely
\[
    \mathcal F^{N,t}_{s}
    :=
    \sigma\bigl(
        B_r-B_t,\,
        W^1_r-W^1_t,\ldots,W^N_r-W^N_t:
        t\leq r\leq s
    \bigr),
    \qquad s\in[t,T].
\]

The set of admissible centralized controls starting from time \(t\) is defined by
\[
    \mathcal A_{N,t}
    :=
    \left\{
    \alpha:[t,T]\times\Omega^{N}\to A :
    \alpha \text{ is } \mathbb F^{N,t}\text{-predictable},
    \displaystyle
    \mathbb E^{\mathbb P^{N}}
    \left[
        \int_t^T \rho(\alpha_s,a_0)^2\,ds
    \right]
    <\infty
    \right\}.
\]
Let
    $x=(x_1,\ldots,x_N)\in(\mathbb T^d)^N$
be the initial configuration. For each \(k=1,\ldots,N\), choose a lift
\(\widehat x_k\in\mathbb R^d\) such that
\(
    q(\widehat x_k)=x_k.
\)

For a given \(\alpha\in\mathcal A_{N,t}\), the state process of the
\(k\)-th particle is defined by
\[
    X^{N,k,t,x,\alpha}_r
    :=
    q\bigl(\widehat X^{N,k,t,x,\alpha}_r\bigr),
    \qquad r\in[t,T],
\]
where for $k=1,\ldots,N$, the lifted process \(\widehat X^{N,k,t,x,\alpha}\) solves
\[
\begin{aligned}
    \widehat X^{N,k,t,x,\alpha}_r
    &=
    \widehat x_k
    +
    \int_t^r
        b\bigl(
            s,
            q(\widehat X^{N,k,t,x,\alpha}_s),
            \mu^{N,t,x,\alpha}_s,
            \alpha_s
        \bigr)\,ds
    +
    \int_t^r
        \sigma\,dW^k_s
    +
    \int_t^r
        \sigma_0\,dB_s,
    \qquad r\in[t,T].
\end{aligned}
\]
The empirical distribution of the torus-valued particle system is given by
\[
    \mu^{N,t,x,\alpha}_r
    :=
    \frac{1}{N}\sum_{k=1}^{N}
    \delta_{X^{N,k,t,x,\alpha}_r}
    \in\mathcal P(\mathbb T^d),
    \qquad r\in[t,T].
\]
Since the coefficients are periodic in the state variable, the torus-valued
processes \(X^{N,k,t,x,\alpha}\) do not depend on the particular choice of
the lifts \(\widehat x_k\).

The associated centralized finite-population control problem is
\begin{align}\label{eq:Nvalue}
    V^N(t,x)
    :=
    \sup_{\alpha\in\mathcal A_{N,t}}
    J_N(t,x,\alpha),
    \qquad (t,x)\in[0,T]\times(\mathbb T^d)^N,
\end{align}
where
\[
    J_N(t,x,\alpha)
    :=
    \mathbb E^{\mathbb P^N}
    \left[
        \int_t^T
        \left\langle
            \ell\bigl(s,\cdot,\mu^{N,t,x,\alpha}_s,\alpha_s\bigr),
            \mu^{N,t,x,\alpha}_s
        \right\rangle\,ds
        +
        G(\mu^{N,t,x,\alpha}_T)
    \right].
\]

\section{Main result}
\label{sec:main-result}
We now state the assumptions needed for our results.
\begin{assumption}[Smooth data and optimized Hamiltonian]\label{ass:rate}
Let \(r:=k^\ast+2\).

\noindent $\mathrm{(i)}$\textbf{Regularity of the coefficients.}
    The maps
    \[
        b:[0,T]\times\mathbb T^d\times\mathcal P(\mathbb T^d)
        \times A\longrightarrow\mathbb R^d,\quad
        \ell:[0,T]\times\mathbb T^d\times\mathcal P(\mathbb T^d)
        \times A\longrightarrow\mathbb R
    \]
    are continuous in \(t\) and belong, uniformly in \(t\), to
    \(C_b^r\) jointly in \((x,\mu,a)\), where differentiability with
    respect to \(\mu\) is understood in the sense of Lions. The derivatives
    with respect to the state variable \(x\) are uniformly bounded.


    \vspace{0.5em}

    \noindent $\mathrm{(ii)}$\textbf{Regularity of the terminal functional.}
    The functional
        $G:\mathcal P(\mathbb T^d)\longrightarrow\mathbb R$
    belongs to \(C_b^r\) in the sense of Lions. Thus, for every
    \(j\geq1\), the derivatives
    $
        D_{\boldsymbol y}^{\boldsymbol\gamma}
        \partial_\mu^jG(\mu)(y_1,\ldots,y_j)
    $
    exist, are jointly continuous, and are uniformly bounded whenever
        $j+\sum_{q=1}^j|\gamma_q|\leq r.$

    \vspace{0.5em}

    \noindent $\mathrm{(iii)}$\textbf{Uniform ellipticity.}
    The constant matrix \(\sigma\) satisfies, for some \(\lambda>0\),
    $
        \sigma\sigma^\top\geq\lambda I_d.
    $ 

\vspace{0.5em}

    \noindent $\mathrm{(iv)}$
    \textbf{Hamiltonian regularity.}
 The Hamiltonian $\mathfrak{H}$ defined in \eqref{eq:hamiltonian}
is continuous in \(t\) and \(C^r\) in the measure variable, in the Lions
sense. Its derivatives, up to total order
\(r\), are jointly continuous in all their variables.
\end{assumption}

\begin{thm}[Particle approximation rate]\label{thm:rate}
    Under Assumption~\ref{ass:rate}, there exists a constant \(C>0\),
    independent of \(N\), such that
    \begin{align*}
        \sup_{(t,x) \in [0,T] \x (\T^d)^N} \big| V^N(t,x) - V(t,\mu^N_x) \big|
        \le
         \left\{
    \begin{aligned}
        &CN^{-\frac{1}{6}}, &~\mbox{if}~ d = 1,
        \\
        &CN^{-\frac{1}{6}}\ln^{\frac{1}{3}}(N), &~\mbox{if}~ d = 2,
        \\
        &CN^{-\frac{1}{3d}}, &~\mbox{if}~ d > 2.
    \end{aligned}
    \right.
    \end{align*}
\end{thm}
\begin{remark}[Comparison with the optimal weak rate in \cite{BouchardTan}]
The estimate of order $N^{-1}$ obtained in
\cite[Theorem~4.5]{BouchardTan} should be viewed as an
optimal weak value-function rate under the special
structure imposed in
\cite[Assumption~4.3]{BouchardTan}. More precisely, both
$\sigma$ and $\sigma_0$ are constant, the controlled drift satisfies
\(
    b(t,x,\mu,a)\equiv a,
\)
and the running reward is of the form
\(
    \ell(t,\mu,a)=\ell_0(t,a)+F(\mu).
\)
Under these assumptions, for a problem starting from $(t,\mu)$, the
conditional law can be represented as
\[
    \mu_s^\alpha
    =
    \tau_{Z_s^\alpha\#}
    \bigl(\mu * \mathcal{N}(0,(s-t)\sigma\sigma^\top)\bigr),
    \qquad s\in[t,T],
\]
where $\tau_z(x)=x+z$ and \(Z^\alpha\) is a finite-dimensional
controlled common translation. Consequently, the mean-field control
problem reduces to a classical finite-dimensional HJB equation,
parametrized by the initial law \(\mu\). This reduction is used to
obtain the required \(C^2\) regularity with respect to the measure
parameter and, subsequently, the \(N^{-1}\) convergence rate. Thus, constant
volatility and the special structure of the drift are central to the
proof in \cite{BouchardTan}, although these assumptions may not be necessary
for the validity of an \(N^{-1}\) weak rate in greater generality.

The present result is complementary. It avoids the above reduction and works
directly with the PDE in the Wasserstein space. It also allows general nonseparable
running rewards and controlled drifts through the regularity assumption on the
optimized common-control Hamiltonian.
Both \(\sigma\) and \(\sigma_0\) remain constant throughout the theorem and
its proof.
\end{remark}

\begin{remark}[State- and measure-dependent common noise in dimension one]
The assumption that \(\sigma_0\) is constant is used in the present
proof through the translation lift
\[
    (t,\mu,z)\longmapsto V\bigl(t,(\mathrm{Id}+z)_\#\mu\bigr),
\]
which converts the constant common-noise Hessian into an ordinary
second derivative in the variable \(z\). In one space dimension, with
scalar common noise, this translation lift can be replaced by the
measure-dependent nonlinear flow introduced in
\cite{BayraktarEkrenHeZhangCommonNoise}. Namely, for
\(\sigma_0=\sigma_0(x,\mu)\), let \(\psi\) solve
\[
    \partial_z\psi(x,\mu,z)
    =
    \sigma_0\Bigl(
        \psi(x,\mu,z),
        \psi(\cdot,\mu,z)_\#\mu
    \Bigr),
    \qquad
    \psi(x,\mu,0)=x,
\]
and define
\[
    \overline V(t,\mu,z)
    :=
    V\bigl(t,\psi(\cdot,\mu,z)_\#\mu\bigr).
\]
Under the periodic analogues of the regularity assumptions in
\cite[Assumption~2.1]{BayraktarEkrenHeZhangCommonNoise}, this flow
transforms the general common-noise Hessian generated by
\(\sigma_0(x,\mu)\) into an ordinary second derivative
\(\partial_{zz}^2\overline V\), with the induced first-order correction
absorbed into the transformed Hamiltonian.

Applying the same flow to the empirical measures in the \(N\)-particle
equation and using the flow derivative estimates, the additional transformed
coefficients and correction terms retain the mean-field scaling used in the
particle residual estimate. Consequently, provided that the transformed
Hamiltonian satisfies the structural conditions of
\cite[Assumption~3.2]{BayraktarEkrenHeZhangCommonNoise}, the comparison
principle in \cite[Theorem~3.1]{BayraktarEkrenHeZhangCommonNoise} allows the
argument of Theorem~\ref{thm:rate} to be extended to state- and
measure-dependent common-noise coefficients in dimension one.
\end{remark}

\subsection{Discussion of the Hamiltonian regularity assumption}
\label{sec:discussion-H}

We discuss the regularity assumption imposed on the optimized
Hamiltonian. Since the state space is the torus \(\mathbb T^d\), all functions
of the state variable are understood as periodic functions on \(\mathbb R^d\).
Derivatives with respect to \(x\in\mathbb T^d\) are computed using periodic
lifts. Assumption~\ref{ass:rate}~(iv) requires joint continuity of the
Hamiltonian derivatives. On sets where the \(p\)-variables and their
evaluation points are restricted to a fixed bounded set, compactness makes
these derivatives bounded and uniformly continuous. The assumption is
motivated by
\cite{DaudinJacksonSeeger,BayraktarEkrenZhangRate}.

In the centralized finite-population problem considered above, the control is
common to all particles. Therefore the optimized Hamiltonian is not, in
general, the average of pointwise optimized Hamiltonians. More precisely, the
relevant Hamiltonian is
\[
\sup_{a\in A}
\frac1N
\sum_{i=1}^N
\left[
\ell(t,x_i,\mu_x^N,a)
+
b(t,x_i,\mu_x^N,a)\cdot p_i
\right],
\]
rather than
\[
\frac1N
\sum_{i=1}^N
\sup_{a\in A}
\left[
\ell(t,x_i,\mu_x^N,a)
+
b(t,x_i,\mu_x^N,a)\cdot p_i
\right].
\]
This distinction is important because the same action \(a\) is chosen for all
particles.

It is therefore convenient to encode this centralized structure through an
optimized Hamiltonian defined on probability measures over the product space
\(E\). Recalling the notation in Subsection~\ref{subsec:notation}, for
\(\eta\in\mathcal P_2(E)\), we define
\[
\mathfrak H(t,\eta)
:=
\sup_{a\in A}
\int_{\mathbb T^d\times\mathbb R^d}
\left[
\ell(t,x,m_\eta,a)+b(t,x,m_\eta,a)\cdot p
\right]\eta(dx,dp).
\]
For
\(x=(x_1,\ldots,x_N)\in(\mathbb T^d)^N\) and
\(p=(p_1,\ldots,p_N)\in(\mathbb R^d)^N\), set
\(
\eta_{x,p}^N
:=
\frac1N
\sum_{i=1}^N
\delta_{(x_i,p_i)}.
\)
Then the Hamiltonian appearing in the \(N\)-particle equation is
\[
\mathfrak H^N(t,x,p)
:=
\mathfrak H(t,\eta_{x,p}^N).
\]
Thus the \(N\)-particle HJB equation can be written, up to the sign convention
used above, as
\[
-\partial_t v^N(t,x)
-
\mathcal L^N v^N(t,x)
-
\mathfrak H^N
\left(
t,x,
N D_{x_1}v^N(t,x),
\ldots,
N D_{x_N}v^N(t,x)
\right)
=0,
\]
where \(\mathcal L^N\) denotes the second-order operator generated by the
idiosyncratic and common noises.

Assumption~\ref{ass:rate}~(iv) implies the following bounds on the empirical
derivatives of \(\mathfrak H^N\). Let
$
z_i=(x_i,p_i)\in E.
$
For every \(R>0\), every \(1\le q\le r\), and every choice of indices
\(i_1,\ldots,i_q\in\{1,\ldots,N\}\), there exists a constant \(C_R\),
independent of \(N\), such that
\[
\left|
D_{z_{i_1}}\cdots D_{z_{i_q}}
\mathfrak H^N(t,z_1,\ldots,z_N)
\right|
\le
C_R
N^{-\#\{i_1,\ldots,i_q\}},
\qquad
\max_{1\le i\le N}|p_i|\le R.
\]
In particular,
\[
\left|
D_{z_i}\mathfrak H^N
\right|
\le
\frac{C_R}{N},
\qquad
\left|
D^2_{z_i z_i}\mathfrak H^N
\right|
\le
\frac{C_R}{N},
\qquad
\left|
D^2_{z_i z_j}\mathfrak H^N
\right|
\le
\frac{C_R}{N^2},
\qquad i\neq j.
\]
These estimates provide precisely the mean-field scaling needed when
differentiating the \(N\)-particle HJB equation. In particular, they are the
key ingredient in showing that the particle-derivative estimates of
\cite{CardaliaguetDaudinJacksonSouganidis,DaudinDelarueJackson}, which do not
directly apply to the optimized Hamiltonian \(\mathfrak H^N\), remain valid in
our framework; the resulting derivative bounds for \(v^N\) are established in
Proposition~\ref{prop:particle-regularity}. 

The assumption is imposed directly on the optimized Hamiltonian. It avoids
requiring uniqueness of the maximizing control. Such a uniqueness assumption is
often unnecessarily strong. Indeed, what is needed in the proof of the
estimates for \(v^N\) is the smoothness of the optimized Hamiltonian itself,
not the smoothness or uniqueness of the optimizer.

The functional $\mathfrak H$ also governs the limiting equation. For
\(\mu\in\mathcal P(\mathbb T^d)\) and a continuous map
\(\phi:\mathbb T^d\longrightarrow\mathbb R^d\), consider the graph measure
\[
\eta_{\mu,\phi}:=(I_d,\phi)_\sharp\mu\in\mathcal P_2(E),
\]
the image of \(\mu\) under \(x\longmapsto(x,\phi(x))\). Since
\(\pi_x\circ(I_d,\phi)=I_d\), its spatial marginal is
\(m_{\eta_{\mu,\phi}}=\mu\), and therefore
\[
\mathfrak H(t,\eta_{\mu,\phi})
=
\sup_{a\in A}
\int_{\mathbb T^d}
\bigl[
\ell(t,x,\mu,a)+b(t,x,\mu,a)\cdot\phi(x)
\bigr]\,\mu(dx).
\]
Taking \(\phi=D_\mu V(t,\mu,\cdot)\), the supremum in
\eqref{eq:original_limit_PDE} is exactly
\(\mathfrak H(t,\eta_{\mu,\phi})\). Hence the single functional
\(\mathfrak H\) is evaluated along graph measures in the limiting equation
and along the empirical measures \(\eta^N_{x,p}\), with
\(p_i=ND_{x_i}v^N(t,x)\), in the \(N\)-particle equation. The regularity
required in Assumption~\ref{ass:rate}~(iv) therefore bears on both equations
at once.

\begin{example}[A bounded periodic-control model]
Let $A=\mathbb{R}$ and fix a constant $c>0$. Let
\[
f,g:[0,T]\times\mathbb{T}^d\longrightarrow\mathbb{R},
\qquad
b_0,B:[0,T]\times\mathbb{T}^d\longrightarrow\mathbb{R}^d,
\]
and let
\(
K:[0,T]\times\mathbb{T}^d\times\mathbb{T}^d
\longrightarrow\mathbb{R}
\)
be symmetric in its spatial variables:
\(
K(t,x,y)=K(t,y,x)
\).
Assume that these functions are continuous in $t$, smooth and periodic in
their spatial variables, and that all spatial derivatives up to order $r$ are
bounded uniformly in $t$. Define
\[
\ell(t,x,m,a)
=
f(t,x)
+
\frac{1}{2}
\int_{\mathbb{T}^d}
K(t,x,y)\,m(dy)
+
c\cos a
+
g(t,x)\sin a,
\]
and
\(
b(t,x,m,a)
=
b_0(t,x)+B(t,x)\sin a.
\)
Since $\sin a$, $\cos a$, and all of their derivatives are bounded, the maps
$b$ and $\ell$ satisfy Assumption~\ref{ass:rate}~(i). In particular, the state
derivatives are uniformly bounded in $a$; for example,
\[
D_x b(t,x,m,a)
=
D_x b_0(t,x)+D_x B(t,x)\sin a.
\]

For $(x,p)\in\mathbb{T}^d\times\mathbb{R}^d$, set
\[
q(t,x,p)
:=
g(t,x)+B(t,x)\cdot p,
\]
and, for $\eta\in\mathcal{P}_2(E)$, define
\[
M(t,\eta)
:=
\int_{\mathbb{T}^d\times\mathbb{R}^d}
q(t,x,p)\,\eta(dx,dp).
\]
Using the symmetry of $K$, the Hamiltonian becomes
\[
\begin{aligned}
\mathfrak H(t,\eta)
={}&
\int_{\mathbb{T}^d\times\mathbb{R}^d}
\left[
f(t,x)+b_0(t,x)\cdot p
\right]
\eta(dx,dp)
\\
&+
\frac{1}{2}
\iint_{(\mathbb{T}^d\times\mathbb{R}^d)^2}
K(t,x,x')
\,\eta(dx,dp)\eta(dx',dp')
\\
&+
\sup_{a\in\mathbb{R}}
\left[
c\cos a+M(t,\eta)\sin a
\right].
\end{aligned}
\]
The last optimization is explicit:
\[
\sup_{a\in\mathbb{R}}
\left[
c\cos a+M\sin a
\right]
=
\sqrt{c^2+M^2}.
\]
Consequently, if
\(
\Psi(s):=\sqrt{c^2+s^2},
\)
then
\[
\begin{aligned}
\mathfrak H(t,\eta)
={}&
\int_{\mathbb{T}^d\times\mathbb{R}^d}
\left[
f(t,x)+b_0(t,x)\cdot p
\right]
\eta(dx,dp)
\\
&+
\frac{1}{2}
\iint_{(\mathbb{T}^d\times\mathbb{R}^d)^2}
K(t,x,x')
\,\eta(dx,dp)\eta(dx',dp')
+
\Psi\left(M(t,\eta)\right).
\end{aligned}
\]
The condition $c>0$ ensures that $\Psi$ is smooth on $\mathbb{R}$. Notice
also that the maximizing action is not unique as an element of
$\mathbb{R}$, since it is determined only modulo $2\pi$, whereas the
optimized Hamiltonian is smooth. This illustrates why
Assumption~\ref{ass:rate}~(iv) is
formulated directly in terms of \(\mathfrak H\).

Writing $z=(x,p)$ and $z'=(x',p')$, the first variational derivative is
\[
\begin{aligned}
\frac{\delta\mathfrak H}{\delta\eta}(t,\eta,z)
={}&
f(t,x)+b_0(t,x)\cdot p
+
\int_{\mathbb{T}^d\times\mathbb{R}^d}
K(t,x,x')\,\eta(dx',dp')
\\
&+
\Psi'\left(M(t,\eta)\right)q(t,x,p),
\end{aligned}
\]
and the second variational derivative is
\[
\frac{\delta^2\mathfrak H}{\delta\eta^2}
(t,\eta,z,z')
=
K(t,x,x')
+
\Psi''\left(M(t,\eta)\right)
q(t,x,p)q(t,x',p').
\]
For every $3\leq j\leq r$, writing \(z_k=(x_k,p_k)\),
\[
\frac{\delta^j\mathfrak H}{\delta\eta^j}
(t,\eta,z_1,\ldots,z_j)
=
\Psi^{(j)}\left(M(t,\eta)\right)
\prod_{k=1}^j q(t,x_k,p_k).
\]
Here
\(
\Psi'(s)
=
\frac{s}{\sqrt{c^2+s^2}},
\,
\Psi''(s)
=
\frac{c^2}{(c^2+s^2)^{3/2}},
\)
and all derivatives of $\Psi$ of positive order are bounded.

Fix $R>0$ and suppose that
\(
\eta\left(\mathbb{T}^d\times B_0(R)\right)=1.
\)
Then $q(t,x,p)$ and $M(t,\eta)$ are uniformly bounded whenever
$|p|\leq R$. The preceding formulas, together with the smoothness of
$f$, $g$, $b_0$, $B$, and $K$, show that the variational derivatives of
$\mathfrak H$ and all their derivatives in the variables $(x,p)$ up to total order
$r$ are bounded and uniformly continuous on the sets appearing in
Assumption~\ref{ass:rate}~(iv). Hence \(\mathfrak H\) satisfies
Assumption~\ref{ass:rate}~(iv).

For the empirical measure $\eta_{x,p}^N$, the Hamiltonian is
\[
\begin{aligned}
\mathfrak H^N(t,x,p)
={}&
\frac{1}{N}
\sum_{i=1}^N
\left[
f(t,x_i)+b_0(t,x_i)\cdot p_i
\right]
+
\frac{1}{2N^2}
\sum_{i,j=1}^N
K(t,x_i,x_j)
\\
&+
\Psi\left(
\frac{1}{N}
\sum_{i=1}^N
\left[
g(t,x_i)+B(t,x_i)\cdot p_i
\right]
\right).
\end{aligned}
\]
The empirical derivative estimates stated above follow directly from this
formula: each distinct particle index introduced by differentiation produces
one factor of $N^{-1}$.

To obtain a complete example satisfying all parts of
Assumption~\ref{ass:rate}, take the
idiosyncratic-noise dimension equal to $d$ and set
\(
\sigma=I_d.
\)
Let $\sigma_0$ be any fixed constant matrix, and define the terminal
functional by
\[
G(m)
=
\int_{\mathbb{T}^d}
\gamma(x)\,m(dx)
+
\frac{1}{2}
\iint_{\mathbb{T}^d\times\mathbb{T}^d}
\Gamma(x,y)\,m(dx)m(dy),
\]
where $\gamma$ and the symmetric kernel $\Gamma$ are smooth and periodic,
with bounded derivatives up to order $r$. Then
$G\in C_b^r(\mathcal{P}(\mathbb{T}^d))$ in the Lions sense, and
\(
\sigma\sigma^\top=I_d.
\)
Thus Assumption~\ref{ass:rate}~(ii) and
Assumption~\ref{ass:rate}~(iii) also hold.
\end{example}

\subsection{Strategy for proving Theorem~\ref{thm:rate}}
\label{subsec:rate-proof-strategy}

The proof of Theorem~\ref{thm:rate} is based on a viscosity comparison
between the mean-field equation and its finite-particle approximation. This
subsection sets up the comparison and identifies the residual estimate proved
below; the proof of the theorem then assembles these estimates by comparison.

The mean-field value function solves the following PDE on
$\mathcal P(\T^d)$:
\begin{align}\label{eq:original_limit_PDE}
        -\pa_t V(t,\mu)& -\mathfrak{H}(t,\eta_{\mu,D_{\mu}V(t,\mu)})
- \frac{1}{2}\int_{\T^d}\text{Tr}\big(\sigma\sigma^\top D_{x\mu} V(t,\mu,x)\big)\,\mu(dx)
        - \frac{1}{2}\text{Tr}(\sigma_0\sigma_0^\top\Hc V(t,\mu))  =0,\notag \\
        V(T,\mu) =&~ G(\mu),
\end{align}
where the Hessian operator $\Hc$ is given by
\begin{align*}
& \Hc V(t,\mu) := \int_{\T^d} D_{x\mu}V(t,\mu,x)\,\mu(dx) ~ + ~\int_{\T^d}\int_{\T^d}D^2_{\mu\mu} V(t,\mu,x,y)\, \mu(dx)\mu(dy).
\end{align*}
The corresponding $N$-particle approximation is given by the following PDE:
\begin{equation}\label{eq:N_particle}
    \begin{split}
        -\pa_tV^N(t,x) &- \mathfrak{H}\Big(t,\frac1N\sum_{i=1}^N\delta_{(x_i,ND_{x_i}V^N(t,x))} \Big)
        \\ &- \frac{1}{2}\sum_{i = 1}^N \text{Tr}\big(\sigma\sigma^\top D^2_{x_i}V^N(t,x)\big)
        - \frac{1}{2}\sum_{i,j = 1}^N \text{Tr}\big(\sigma_0\sigma_0^\top D^2_{x_i x_j}V^N(t,x)\big)
        ~ = ~ 0,
        \\
        V^N(T,x)~ &=~ G(\mu^N_x),
    \end{split}
\end{equation}
where, for each $N \in \N_+$,
$x := (x_i)_{i \in [N]} \in (\T^d)^N$ and
    $\mu^N_x := \frac{1}{N}\sum_{i = 1}^N\delta_{x_i}.$

For \(\mu\in\mathcal P(\mathbb T^d)\) and \(z\in\mathbb T^d\), write
\(\mu^z:=(I_d+z)_\sharp\mu\). Define the extended value function
$\Vb:[0,T] \x \Pc(\T^d)\x \T^d \longrightarrow \R$ by
\begin{align*}
    \Vb(t,\mu,z) := V(t,\mu^z),
    ~\mbox{for}~(t,\mu,z) \in [0,T] \x \Pc(\T^d)\x \T^d,
\end{align*}
and observe that
$\pa^2_{zz}\Vb(t,\mu,z) = \Hc V(t,\mu^z)$.
If $V$ solves \eqref{eq:original_limit_PDE}, then $\Vb$ solves the
PDE
\begin{equation}\label{eq:extendedPDE}
    \begin{split}
        -\pa_t \Vb(t,\mu,z) ~-& ~
        \sup_{a \in A}\int_{\T^d}
        \Big(
        \ell(t,x+z,\mu^z,a)
        +
        b(t,x+z,\mu^z,a)\cdot D_\mu \Vb(t,\mu, z, x)
        \Big)\mu(dx)
        \\ &
- \frac{1}{2}\int_{\T^d}
\text{Tr}\big(\sigma\sigma^\top 
D_{x\mu} \Vb(t,\mu,z, x)\big)\,\mu(dx)
        \\ &
        - \frac{1}{2}\text{Tr}(\sigma_0\sigma_0^\top\pa^2_{zz}\Vb(t,\mu,z))  =0,
        \\
        \Vb(T,\mu,z) =&~ G(\mu^z).
    \end{split}
\end{equation}
Similarly, define the extended value functions
$\Vb^N: [0,T]\x \R^{dN} \x \T^d \longrightarrow \R$ by
$
    \Vb^N(t,x,z) := V^N(t,x+z),
$
where $x+z := (x_i + z)_{i \in [N]}$; in particular,
\(\mu^N_{x+z}=(\mu_x^N)^z\).

\noindent \textbf{Step 1. Define the inf-convolution $\Vb^N_\eps$ of $\Vb^N$.} Let $V^N$ be the unique classical solution to \eqref{eq:N_particle}.
With the notation above, define the inf-convolution of \(\Vb^N\) by
\begin{align*}
    \Vb^N_\eps(t,\mu,z) := \inf_{s,x,y \in [0,T] \x \R^{dN} \x \T^d} \bigg(\Vb^N(s,x,y) + \frac{|t-s|^2 + |z-y|^2 + \rho_F^2(\mu,\mu^N_x)}{2\eps} \bigg).
\end{align*}
We will prove that $\Vb^N_\eps$ is a viscosity supersolution to \eqref{eq:extendedPDE} up to an error term $E(N,\eps)$. Let $\Phi : [0,T] \x \Pc(\T^d) \x \T^d \longrightarrow \R$ be partially
$C^2$-regular, and assume that
$(t_0,\mu_0,z_0) \in [0,T]\x \Pc(\T^d) \x \T^d$ is a strict local minimum
point of $\Vb^N_\eps - \Phi$ with $t_0 < T$.

\vspace{0.5em}

\noindent\textbf{Step 2. Use the transformed particle equation.} Define the super-convolution
$\Phi_\eps : [0,T] \x \R^{dN} \x \T^d \longrightarrow \R$ of $\Phi$ by
\begin{align*}
    \Phi_\eps(s,x,w) := \sup_{z \in \T^d}\bigg(\Phi(t_0,\mu_0,z) - \frac{|w-z|^2 + |t_0 - s|^2 + \rho_F^2(\mu^N_x,\mu_0)}{2\eps}\bigg).
\end{align*}
Then $\Vb^N - \Phi_\eps$ attains a local minimum at
$(s_0,x_0,w_0) \in [0,T) \x \R^{dN} \x \T^d$,
where $(s_0,x_0,w_0)$ satisfies
\begin{align*}
    \Vb^N_\eps(t_0,\mu_0,z_0)
    =
    \Vb^N(s_0,x_0,w_0) + \frac{|t_0-s_0|^2 + |z_0-w_0|^2 + \rho_F^2(\mu_0,\mu^N_{x_0})}{2\eps}.
\end{align*}
The common-noise term in the transformed particle equation is
\(\frac12\operatorname{Tr}(\sigma_0\sigma_0^\top D^2_{ww}\Vb^N)\).
Set
\[
m_N:=\mu^N_{x_0},
\qquad
p(x):=D_\mu\Phi(t_0,\mu_0,z_0)(x),
\qquad
\Sigma:=\sigma\sigma^\top.
\]
For \(\lambda,\nu\in\mathcal P(\mathbb T^d)\) and a continuous map
\(q:\mathbb T^d\to\mathbb R^d\), define
\begin{align*}
\mathcal K(t,z;\lambda,\nu,q)
:={}&
\sup_{a\in A}\int_{\mathbb T^d}
\Big[
\ell(t,x+z,\nu^z,a)
+b(t,x+z,\nu^z,a)\cdot q(x)
\Big] \,\lambda(dx).
\end{align*}
Thus the limiting Hamiltonian at the contact point is
\(\mathcal K(t_0,z_0;\mu_0,\mu_0,p)\), whereas the particle Hamiltonian is
\(\mathcal K(s_0,w_0;m_N,m_N,p)\).

For \(\varepsilon\) sufficiently small,
Lemma~\ref{lem:derivative_calcu} shows that \(\Phi_\varepsilon\) is \(C^2\)
in \(w\). Since \(\Vb^N\) solves the transformed particle equation classically,
the local minimality of \(\Vb^N-\Phi_\eps\) at \((s_0,x_0,w_0)\), together
with
\(
D^2_{ww}\Phi_\varepsilon(s_0,x_0,w_0)
\ge D^2_{zz}\Phi(t_0,\mu_0,z_0),
\)
yields
\begin{align}
-\partial_t\Phi(t_0,\mu_0,z_0)
\ge{}&
\mathcal K(s_0,w_0;m_N,m_N,p)
+\frac12\sum_{i=1}^N
\operatorname{Tr}\!\left(
\Sigma D^2_{x_i}\Phi_\varepsilon(s_0,x_0,w_0)
\right)
\notag\\
&+\frac12\operatorname{Tr}\!\left(
\sigma_0\sigma_0^\top D^2_{zz}\Phi(t_0,\mu_0,z_0)
\right).
\label{eq:particle-contact-inf}
\end{align}

\vspace{0.5em}

\noindent\textbf{Step 3. Define and estimate the residuals.}
For the inf-convolution argument, let
\begin{align*}
g(x)&:=\operatorname{Tr}\!\left(
\Sigma D_{x\mu}\Phi(t_0,\mu_0,z_0)(x)
\right),\\
Q_N(x,y) & 
:=D^2_{\mu\mu}
\!\left[\mu\mapsto\rho_F^2(\mu,\mu_0)\right]
(m_N)(x,y).
\end{align*}
Define the supersolution residual by the
telescoping formula
\begin{align}
E(N,\eps):={}&
\mathcal K(t_0,z_0;m_N,m_N,p)
-\mathcal K(s_0,w_0;m_N,m_N,p)
\notag\\
&+\mathcal K(t_0,z_0;\mu_0,m_N,p)
-\mathcal K(t_0,z_0;m_N,m_N,p)
\notag\\
&+\mathcal K(t_0,z_0;\mu_0,\mu_0,p)
-\mathcal K(t_0,z_0;\mu_0,m_N,p)
\notag\\
&+\frac12\int_{\mathbb T^d}g(x)\,(\mu_0-m_N)(dx)
+\frac{1}{4\eps N^2}\sum_{i=1}^N
\operatorname{Tr}\!\left(
\Sigma Q_N(x_{0,i},x_{0,i})
\right).
\label{eq:error-inf}
\end{align}
The first three lines separate, respectively, the time--translation error, the
error from replacing the empirical averaging measure by \(\mu_0\), and the
error from replacing the measure argument in the coefficients. By
Lemma~\ref{lem:derivative_calcu}, the terms telescope exactly to
\begin{align*}
E(N,\eps)
={}&
\mathcal K(t_0,z_0;\mu_0,\mu_0,p)
+\frac12\int_{\mathbb T^d}g(x)\,\mu_0(dx)
\\
&-\mathcal K(s_0,w_0;m_N,m_N,p)
-\frac12\sum_{i=1}^N
\operatorname{Tr}\!\left(
\Sigma D^2_{x_i}\Phi_\varepsilon(s_0,x_0,w_0)
\right).
\end{align*}
Consequently, after adding the residual \(|E(N,\eps)|\),
\eqref{eq:particle-contact-inf} yields
\begin{align*}
-\partial_t\Phi(t_0,\mu_0,z_0)+|E(N,\eps)|
\ge{}&
\mathcal K(t_0,z_0;\mu_0,\mu_0,p)
+\frac12\int_{\mathbb T^d}g(x)\,\mu_0(dx)
+\frac12\operatorname{Tr}\!\left(
\sigma_0\sigma_0^\top D^2_{zz}\Phi(t_0,\mu_0,z_0)
\right).
\end{align*}
This shows that $\Vb^N_\eps$ satisfies the supersolution inequality for
\eqref{eq:extendedPDE} up to the error $|E(N,\eps)|$. Together with the
terminal estimate $\Vb^N_\eps(T,\mu,z)\ge G(\mu^z)-C\eps^{1/3}$, comparison
yields
\[
\Vb\;\le\;\Vb^N_\eps+|E(N,\eps)|(T-t)+C\eps^{1/3}
\qquad\text{on }[0,T]\times\Pc(\T^d)\times\T^d.
\]

For the reverse inequality, define the sup-convolution
\begin{align*}
\Vb^{N,\eps}(t,\mu,z)
:=\sup_{s,x,w}
\left\{
\Vb^N(s,x,w)
-\frac{|t-s|^2+|z-w|^2+\rho_F^2(\mu,\mu_x^N)}{2\eps}
\right\}.
\end{align*}
Similarly, $\Vb^{N,\eps}$ is a viscosity subsolution up to an error
$|E(N,\eps)|$, with $\Vb^{N,\eps}(T,\mu,z)\le G(\mu^z)+C\eps^{1/3}$, and
comparison gives, on $[0,T]\times\Pc(\T^d)\times\T^d$,
\begin{align*}
\Vb^{N,\eps}-|E(N,\eps)|(T-t)-C\eps^{1/3}
\;\le\;\Vb\;\le\;
\Vb^{N}_{\eps}+|E(N,\eps)|(T-t)+C\eps^{1/3}.
\end{align*}
Proposition~\ref{prop:Key_estimate} establishes that, for
$\eps=\alpha(N)$ as in \eqref{eq:alpha},
\begin{align*}
|E(N,\eps)|
\le
\begin{cases}
CN^{-1/6}, & d=1,\\[1mm]
CN^{-1/6}(\ln N)^{1/3}, & d=2,\\[1mm]
CN^{-1/(3d)}, & d>2.
\end{cases}
\end{align*}
Finally, take $\mu=\mu^N_x$ and $z=0$. The zero-penalty competitor
$(s,y,w)=(t,x,0)$ in the definitions of the two convolutions gives
\[
\Vb^N_\eps(t,\mu^N_x,0)\le V^N(t,x)\le\Vb^{N,\eps}(t,\mu^N_x,0),
\]
so the previous inequalities yield
\begin{align*}
&V^N(t,x)-C\big(|E(N,\eps)|(T-t)+\eps^{1/3}\big)
\le\Vb^{N,\eps}(t,\mu^N_x,0)-C\big(|E(N,\eps)|(T-t)+\eps^{1/3}\big)
\le V(t,\mu^N_x)\\
&\le\Vb^N_\eps(t,\mu^N_x,0)+C\big(|E(N,\eps)|(T-t)+\eps^{1/3}\big)
\le V^N(t,x)+C\big(|E(N,\eps)|(T-t)+\eps^{1/3}\big),
\end{align*}
which, since $\eps^{1/3}=\alpha(N)^{1/3}$ is of the same order as the bound
on $|E(N,\eps)|$, proves Theorem~\ref{thm:rate}.

\section{Proof of Theorem~\ref{thm:rate}}
\label{sec:proofs}

 \subsection{Derivatives on the space of probability measures}

    \begin{definition}
        \begin{enumerate}[(i)]
            \item A function $u: \Pc(\T^d) \longrightarrow \R$ has a linear functional derivative if there exists
                \(
                    \delta_{\mu}u: \Pc(\T^d) \x \T^d \to \R
                \)
            such that $\delta_\mu u$ is continuous in the product topology and
                \begin{itemize}
                    \item for each $\mu \in \Pc(\T^d)$, the map
                    $x \mapsto \delta_{\mu} u(\mu,x)$ belongs to
                    $C(\T^d)$,
                    \item for all $m_1, m_2 \in \Pc(\T^d)$,
                        \begin{equation*}
                            u(m_1) - u(m_2)
                            ~ = ~
                            \int_0^1 \int_{\T^d}
                                \delta_{\mu} u(\lambda m_1 + (1 - \lambda)m_2,x)(m_1 - m_2)(dx) d\lambda.
                        \end{equation*}
                \end{itemize}
            \item A function $u: \Pc(\T^d) \longrightarrow \R$ has a
            second-order linear functional derivative if, for any
            $x \in \mathbb{T}^d$, the map
            $\mu \mapsto \delta_{\mu}u$ has a linear functional derivative,
            denoted by $\delta_{\mu}^2 u (\mu)(x,y)$, and
            $(\mu,x,y) \mapsto \delta^2_{\mu} u(\mu,x,y)$ is continuous in the
            product topology.
            \item A function $u: \mathcal{P}(\T^d) \longrightarrow \R$ is
            second-order Lions differentiable if the derivatives
            \[
            (\pa_x\delta_\mu u, \pa^2_{xx}\delta_\mu u,
            \pa_x\pa_y\delta^2_\mu u)
            \]
            belong, respectively, to
            \[
            C^d(\mathcal{P}(\mathbb T^d) \x \T^d),\quad
            C^{d\x d}(\mathcal{P}(\mathbb T^d) \x \T^d),\quad
            C^{d\x d}(\mathcal{P}(\mathbb T^d) \x \T^d \x \T^d).
            \]
            We denote $D_\mu u := \pa_x\delta_\mu u$,
            $D_{x\mu} u := \pa^2_{xx}\delta_\mu u$, and
            $D^2_{\mu\mu} u:= \pa_x\pa_y\delta^2_\mu u$.
        \end{enumerate}
    \end{definition}
    For any second-order Lions differentiable function
    $\mu \longmapsto u(\mu)$, we define the partial Hessian 
    \begin{align*}
        \Hc u :=& \pa^2_{zz}u((I_d + z)_\sharp \mu)|_{z = 0}
        =
        \displaystyle
            \int_{\T^d}\int_{\T^d}D^2_{\mu\mu}u(\mu,x,y)\mu(dx)\mu(dy)
                + \int_{\T^d} D_{x\mu} u(\mu,x)\mu(dx).
    \end{align*}

    \subsection{Viscosity solutions}
    We introduce second-order jets and the corresponding notion of viscosity
    solution to \eqref{eq:original_limit_PDE} as follows.
    \begin{definition}
         We say that a function
         $\ub:\Theta \longrightarrow \R$ is partially $C^2$-regular if
         $\ub \in C(\Theta)$, for each $\mu \in \mathcal{P}(\T^d)$ the map
         $(t, z) \longmapsto \ub(t,\mu, z)$ belongs to \(C^{1,2}\), and the following 
         derivatives exist
         \[
            (D_\mu \ub, D_{x\mu} \ub) \in
            C^d(\Theta \x \T^d) \x C^{d \x d}(\Theta \x \T^d).
         \]
    
    \end{definition}

    \begin{definition}
        Let $\ub:\Theta \longrightarrow \R$ be a locally bounded function, and
        let $\theta \in [0,T) \x \Pc(\T^d) \x \T^d$.
        We define the (partial) second-order superjet
        $J^{2,+}\ub(\theta) \subset \R  \x C^d(\Theta \x \T^d) \x C^{d \x d}(\Theta \x \T^d) \x \S_{d}$ of $\ub$ at $\theta$ by
        \begin{align*}
            J^{2,+}\ub(\theta)
            :=
            \{(\pa_t\phi,  D_\mu \phi, D_{x\mu}\phi, \pa^2_{mm}\phi)(\theta): ~
            & \ub - \phi~\mbox{has a local maximum at}~\theta, \\
            & \phi ~\mbox{is partially}~C^2\mbox{-regular}\},
        \end{align*}
        the second-order subjet $J^{2,-}\ub(\theta) := -J^{2,+}(-\ub)(\theta)$,
        and the closed superjet by
        \begin{align*}
            \Jb^{2,+}\ub(\theta)
            := &
            \{(b, p, q, X) \in \R  \x C^d(\Theta \x \T^d) \x C^{d \x d}(\Theta \x \T^d) \x \S_{d}: \\
            & \ub - \phi^n~\mbox{has a local maximum at}~\theta_n,~
             \phi^n~\mbox{is partially}~C^2\mbox{-regular}; ~\mbox{as}~n \longrightarrow +\infty, \\
             & (\pa_t\phi^n,D_\mu \phi^n, D_{x\mu}\phi^n, \pa^2_{mm}\phi^n)(\theta_n) \longrightarrow (b, p, q, X),
             ~\theta_n \longrightarrow \theta,
             \ub(\theta_n) \longrightarrow  \ub(\theta) \},
        \end{align*}
        and $\Jb^{2,-}\ub(\theta) := - \Jb^{2,+}(-\ub)(\theta) $.
    \end{definition}

    \begin{definition}
        We say that $u$ is a viscosity subsolution (supersolution) to
        \eqref{eq:original_limit_PDE} if its extension
        \[
        \bar u(t,\mu,z)
        :=
        u(t,\mu^z),
        \qquad
        (t,\mu,z)\in[0,T]\x\Pc(\T^d)\x\T^d,
        \]
        is locally bounded and, for $(t,\mu, z) \in [0,T) \x \Pc(\T^d) \x \T^d$
        and
        \[
        (c, p, q, X) \in J^{2,+}\bar u(t,\mu, z)
        \quad
        \bigl(\text{respectively }J^{2,-}\bar u(t,\mu, z)\bigr),
        \]
        \begin{align*}
            - c ~-& ~ \sup_{a \in A}\int_{\T^d} \Big(\ell(t,x+z,\mu^z,a) + b(t,x+z,\mu^z,a)\cdot p(x)\Big)\mu(dx)
        \\ &- \frac{1}{2}\int_{\T^d}\text{Tr}\big(\sigma\sigma^\top q(x)\big)\,\mu(dx)
        - \frac{1}{2}\text{Tr}(\sigma_0\sigma_0^\top X) \le (\ge) ~0.
        \end{align*}
        \end{definition}

\subsection{Preliminaries}
We impose the following additional assumption to ensure comparison for
\eqref{eq:original_limit_PDE}; it is adapted from
\cite{DaudinJacksonSeeger,BayraktarEkrenZhangComparison,
BayraktarEkrenHeZhangComparison,BayraktarEkrenHeZhangCommonNoise,
BayraktarEkrenZhangRate}.
\begin{assumption}[Lipschitz data for comparison]\label{ass:comparison}
The functions $b$ and $\ell$ are Lipschitz continuous on their domains uniformly
in $a$.
\end{assumption}
\begin{prop}[Comparison for the limiting equation]\label{prop:comparison-limit}
    Under Assumptions~\ref{ass:rate} and~\ref{ass:comparison}, the value
    function $V$ is the unique viscosity solution of the mean-field PDE
    \eqref{eq:original_limit_PDE}.
\end{prop}
\begin{proof}[Proof sketch]
The dynamic programming principle and viscosity characterization follow from
standard dynamic-programming arguments for partially observed control in the
filtering formulation; see
\cite{BensoussanPartialObservation,BandiniCossoFuhrmanPham2018,
BandiniCossoFuhrmanPham2019,BayraktarCossoPham2018,DjetePossamaiTanDPP}.
After the translation lift, the common-noise term becomes the ordinary
\(z\)-Hessian, and comparison follows from
\cite{DaudinJacksonSeeger,BayraktarEkrenZhangComparison,
BayraktarEkrenHeZhangComparison,BayraktarEkrenHeZhangCommonNoise,
BayraktarEkrenZhangRate}.
\end{proof}

\begin{prop}[Uniform particle regularity]\label{prop:particle-regularity}
    Under Assumption~\ref{ass:rate}, equation \eqref{eq:N_particle} has a
    unique classical solution $v^N$, which coincides with the value function
    $V^N$ defined in \eqref{eq:Nvalue}. Moreover, for the integer $k^\ast$
    fixed in Subsection~\ref{subsec:notation}, there is a constant $C>0$,
    independent of $N$, such that, for all $N\in\N_+$,
    $1\leq k\leq k^\ast$, and $i_1,\ldots,i_k\in[N]$,
\[
\left|
D_{x_{i_1}}\cdots D_{x_{i_k}} v^N(t,x)
\right|
\leq
\frac{C}{N},
\qquad
\forall (t,x) \in [0,T] \times \mathbb{T}^{dN}.
\]
In addition, for all $0 \leq s < t \leq T$ and
$x,y \in \mathbb{T}^{dN}$,
\[
\left|
v^N(t,x) - v^N(s,y)
\right|
\leq
C
\left(
\sqrt{t-s}
+
W_1(\mu_x^N,\mu_y^N)
\right).
\]
\end{prop}

\begin{proof}
The proof follows the adjoint method for particle approximations of
Wasserstein HJB equations
\cite{CardaliaguetSouganidis,CardaliaguetDaudinJacksonSouganidis,
DaudinDelarueJackson,CardaliaguetJacksonMimikosSouganidis,
BayraktarEkrenZhangRate}. We spell out the adaptation to the present
common-control Hamiltonian.

For \(\varphi\in C^2(\mathbb T^{dN})\), define the generator
\[
\mathcal L^N \varphi
:=
\frac12
\sum_{i=1}^N
\operatorname{Tr}(\sigma\sigma^\top D^2_{x_i x_i}\varphi)
+
\frac12
\sum_{i,j=1}^N
\operatorname{Tr}(\sigma_0\sigma_0^\top D^2_{x_i x_j}\varphi).
\]
Set
\(
\mathcal H^N(t,x,p_1,\ldots,p_N):=\mathfrak H(t,\eta_{x,p}^N)
\). Then equation \eqref{eq:N_particle} becomes
\[
-\partial_t v^N-\mathcal L^N v^N
-\mathcal H^N(t,x,ND_{x_1}v^N,\ldots,ND_{x_N}v^N)=0,
\qquad
v^N(T,x)=G(\mu_x^N).
\]

Let \(\chi_R:\mathbb R^d\to\mathbb R^d\) be a smooth truncation which is the
identity on \(B_0(R)\), with
\(|\chi_R(p)|\leq |p|\) and \(\|D\chi_R\|_\infty\leq C\), uniformly in \(R\),
and set
\[
\mathcal H^{N,R}(t,x,p_1,\ldots,p_N)
:=\mathcal H^N(t,x,\chi_R(p_1),\ldots,\chi_R(p_N)).
\]
The resulting equation is uniformly parabolic on the compact manifold
\(\mathbb T^{dN}\), so it has a classical solution \(v^{N,R}\) by standard
parabolic theory \cite{FriedmanParabolic}.

The terminal condition has the correct mean-field scaling. By the Lions chain
rule,
\[
D_{x_i}G(\mu_x^N)=\frac1N D_\mu G(\mu_x^N,x_i),
\qquad
\left|
D_{x_{i_1}}\cdots D_{x_{i_m}}G(\mu_x^N)
\right|
\le \frac{C}{N},
\quad 1\le m\le k^\ast .
\]
We next obtain the first-derivative estimate independently of \(R\).
For \(a\in A\), write
\[
F_a^N(t,x,q)
:=\frac1N\sum_{k=1}^N
\left[
\ell(t,x_k,\mu_x^N,a)
+b(t,x_k,\mu_x^N,a)\cdot q_k
\right],
\qquad
\mathcal H^N(t,x,q)=\sup_{a\in A}F_a^N(t,x,q).
\]
Let \(x^{i,h}\) be obtained from \(x\) by replacing \(x_i\) with \(x_i+h\).
Since
\(
W_1(\mu_{x^{i,h}}^N,\mu_x^N)\leq |h|/N,
\)
Assumption~\ref{ass:rate} (i) gives, uniformly in \(a\),
\[
\left|F_a^N(t,x^{i,h},q)-F_a^N(t,x,q)\right|
\leq
\frac{C|h|}{N}
\left(1+|q_i|+\frac1N\sum_{j=1}^N|q_j|\right).
\]
Moreover, boundedness of \(b\) makes \(F_a^N\), uniformly in \(a\),
\(C/N\)-Lipschitz in each \(q_i\).
Using
\(
|\sup_a f_a-\sup_a g_a|\leq\sup_a|f_a-g_a|,
\)
Assumption~\ref{ass:rate} (iv), and the properties of
\(\chi_R\), we obtain
\[
\left|D_{x_i}\mathcal H^{N,R}(t,x,p)\right|
\leq \frac{C}{N}
\left(1+|p_i|+\frac1N\sum_{j=1}^N|p_j|\right),
\qquad
\left|ND_{p_i}\mathcal H^{N,R}(t,x,p)\right|\leq C.
\]
These estimates hold for all \(p\), with constants independent of \(R\).

Set \(M_R(t):=\max_{1\le j\le N}\|D_{x_j}v^{N,R}(t)\|_\infty\).
For a scalar component \(u_i\) of \(D_{x_i}v^{N,R}\), differentiation of the
truncated equation gives
\[
-\partial_tu_i-\mathcal L^Nu_i
-\sum_{j=1}^N\beta_j\cdot D_{x_j}u_i
=D_{x_i}\mathcal H^{N,R},
\qquad
\beta_j:=ND_{p_j}\mathcal H^{N,R},
\]
where the Hamiltonian and its derivatives are evaluated at
\((t,x,ND_{x_1}v^{N,R},\ldots,ND_{x_N}v^{N,R})\). The preceding estimates
give \(|\beta_j|\leq C\) and
\(\left|D_{x_i}\mathcal H^{N,R}\right|\leq C/N+CM_R(t)\).
The maximum principle therefore yields
\[
M_R(t)\le M_R(T)+\int_t^T\left(\frac{C}{N}+CM_R(r)\right)dr .
\]
Since \(M_R(T)\le C/N\), Gronwall's lemma gives a constant \(C_0\),
independent of \(N\) and \(R\), such that
\(
\max_{1\le i\le N}\|D_{x_i}v^{N,R}\|_\infty\le \frac{C_0}{N}.
\)
Choosing \(R>C_0\), the truncation is inactive. Thus \(v^{N,R}\) solves the
original equation; uniqueness follows from comparison, and the dynamic
programming principle and standard verification identify it with \(V^N\).
We henceforth denote it by \(v^N\).

In particular, all \(p_i=ND_{x_i}v^N\) lie in a fixed bounded set.
Assumption~\ref{ass:rate} (iv), compactness of
\(\mathcal P(\mathbb T^d\times\overline{B_0(C_0)})\), and the empirical-measure
chain rule now imply that, along the solution and with \(z_i=(x_i,p_i)\),
for \(1\leq q\leq k^\ast\),
\[
\left|
D_{z_{i_1}}\cdots D_{z_{i_q}}\mathcal H^N(t,z_1,\ldots,z_N)
\right|
\le
C N^{-\#\{i_1,\ldots,i_q\}}.
\]

For higher derivatives, we use a simultaneous induction on the derivative
bounds and the adjoint source estimates introduced below. The source estimate
for first derivatives follows from the preceding bounds.
Fix $2\leq m\leq k^\ast$, assume that both estimates hold through order
$m-1$, and let
\(
Z=D_{x_{i_1}}\cdots D_{x_{i_m}}v^N.
\)
After differentiating the equation \(m\) times, \(Z\) solves, componentwise, a
linear backward parabolic equation of the form
\[
-\partial_t Z-\mathcal L^N Z
-\sum_{j=1}^N\beta_j\cdot D_{x_j}Z=R,
\]
where
\(
\beta_j
=
N D_{p_j}\mathcal H^N
\bigl(t,x,ND_{x_1}v^N,\ldots,ND_{x_N}v^N\bigr).
\)
The vector fields $\beta_j$ are uniformly bounded. The source $R$ consists of
lower-order terms and products containing spatial gradients of lower-order
derivatives. Although the latter gradients can have order $m$, they are
controlled by the weighted energy estimate below.

Fix \(s<T\), and let \(\rho\) be a smooth nonnegative function on
\(\mathbb T^{dN}\) such that
\(
\int_{\mathbb T^{dN}}\rho\,dx=1.
\)
Let \(\phi\) solve the forward adjoint equation
\[
\begin{cases}
\partial_r\phi-\mathcal L^N\phi
+\sum_{j=1}^N
\operatorname{div}_{x_j}(\beta_j\phi)=0,
& (r,x)\in(s,T)\times\mathbb T^{dN},\\[1mm]
\phi(s,\cdot)=\rho.
\end{cases}
\]
Since \(\mathcal L^N\) has constant coefficients, it is self-adjoint on the
torus. Moreover, the maximum principle and conservation of mass give
\[
\phi\ge0,
\qquad
\int_{\mathbb T^{dN}}\phi(r,x)\,dx=1,
\qquad s\le r\le T.
\]

We shall also use the corresponding weighted energy estimate. Let
$Y=D_{x_{j_1}}\cdots D_{x_{j_\ell}}v^N$ with $1\leq\ell<m$. Its differentiated
equation has the form
\[
-\partial_rY-\mathcal L^NY
-\sum_{j=1}^N\beta_j\cdot D_{x_j}Y=R_Y.
\]
The simultaneous induction hypothesis gives
\[
\|Y\|_\infty\le\frac{C}{N},
\qquad
\int_s^T\int_{\mathbb T^{dN}}|R_Y|\phi\,dx\,dr
\le\frac{C}{N}.
\]
The second estimate is uniform over $s<T$, all smooth nonnegative unit-mass
initial data $\rho$, and the corresponding adjoint solutions $\phi$.
For notational simplicity, consider one scalar component of \(Y\). Using the
definition of \(\mathcal L^N\), we have
\[
\begin{aligned}
&-\partial_r\frac{|Y|^2}{2}
-\mathcal L^N\frac{|Y|^2}{2}
-\sum_{j=1}^N
 \beta_j\cdot D_{x_j}\frac{|Y|^2}{2}
 =
YR_Y
-\frac12\sum_{j=1}^N
  |\sigma^\top D_{x_j}Y|^2
-\frac12
  \left|
  \sigma_0^\top\sum_{j=1}^ND_{x_j}Y
  \right|^2.
\end{aligned}
\]
Multiplying by \(\phi\), integrating over
\((s,T)\times\mathbb T^{dN}\), and using the adjoint equation gives
\[
\begin{aligned}
&\frac12\int_{\mathbb T^{dN}}|Y(s,x)|^2\rho(x)\,dx
+
\frac12\int_s^T\int_{\mathbb T^{dN}}
\left(
\sum_{j=1}^N|\sigma^\top D_{x_j}Y|^2
+
\left|
\sigma_0^\top\sum_{j=1}^ND_{x_j}Y
\right|^2
\right)\phi\,dx\,dr
\\
&=
\frac12\int_{\mathbb T^{dN}}
|Y(T,x)|^2\phi(T,x)\,dx
+
\int_s^T\int_{\mathbb T^{dN}}YR_Y\phi\,dx\,dr .
\end{aligned}
\]
The common-noise term on the left is nonnegative, while the uniform
ellipticity assumption
\(\sigma\sigma^\top\ge\lambda I_d\) gives
\(
\sum_{j=1}^N|\sigma^\top D_{x_j}Y|^2
\ge
\lambda\sum_{j=1}^N|D_{x_j}Y|^2.
\)
Consequently,
\[
\begin{aligned}
\int_s^T\int_{\mathbb T^{dN}}
\sum_{j=1}^N|D_{x_j}Y|^2\phi\,dx\,dr
&\le
C\|Y(T,\cdot)\|_\infty^2
+
C\|Y\|_\infty
\int_s^T\int_{\mathbb T^{dN}}|R_Y|\phi\,dx\,dr
\le
\frac{C}{N^2}.
\end{aligned}
\]
This estimate absorbs the factors \(N\) created by differentiating
\(ND_{x_j}v^N\). For example, for any two such lower-order derivatives
$Y_1$ and $Y_2$, Cauchy--Schwarz gives
\[
\begin{aligned}
&N\int_s^T\int_{\mathbb T^{dN}}
\sum_{j=1}^N
|D_{x_j}Y_1|\,|D_{x_j}Y_2|\phi\,dx\,dr
\\
&\quad\le
N
\left(
\int_s^T\int_{\mathbb T^{dN}}
\sum_{j=1}^N|D_{x_j}Y_1|^2\phi\,dx\,dr
\right)^{1/2}
\left(
\int_s^T\int_{\mathbb T^{dN}}
\sum_{j=1}^N|D_{x_j}Y_2|^2\phi\,dx\,dr
\right)^{1/2}
\le\frac{C}{N}.
\end{aligned}
\]
All the remaining terms in \(R\) either contain an explicit empirical-measure
factor \(1/N\), or are controlled by the same combination of the induction
hypothesis and the preceding energy estimate. Hence, uniformly over the
adjoint solutions described above,
\(
\int_s^T\int_{\mathbb T^{dN}}|R|\phi\,dx\,dr
\le\frac{C}{N}.
\)

Finally, multiplying the equation for \(Z\) by \(\phi\) and using the adjoint
equation yields
\[
\int_{\mathbb T^{dN}}Z(s,x)\rho(x)\,dx
=
\int_{\mathbb T^{dN}}Z(T,x)\phi(T,x)\,dx
+
\int_s^T\int_{\mathbb T^{dN}}R\phi\,dx\,dr.
\]
Since \(\phi(T,\cdot)\) is a probability density and
\(\|Z(T,\cdot)\|_\infty\le C/N\), we obtain
\[
\left|
\int_{\mathbb T^{dN}}Z(s,x)\rho(x)\,dx
\right|
\le
\|Z(T,\cdot)\|_\infty
+
\int_s^T\int_{\mathbb T^{dN}}|R|\phi\,dx\,dr
\le\frac{C}{N}.
\]
Taking \(\rho\) to approximate a Dirac mass at an arbitrary point gives
\(
\|Z(s,\cdot)\|_\infty\le\frac{C}{N}.
\)
This derivative bound, together with the preceding source estimate for $R$,
closes the simultaneous induction.

It remains to prove the modulus estimate. The equation and terminal condition
are invariant under particle permutations, so \(v^N\) is symmetric. Choosing an
optimal matching between \(\mu_x^N\) and \(\mu_y^N\) and using the first
derivative estimate gives
\[
|v^N(t,x)-v^N(t,y)|\le C W_1(\mu_x^N,\mu_y^N).
\]
Finally, the dynamic programming principle and boundedness of the coefficients
give
\[
|v^N(t,x)-v^N(s,x)|
\le
C(t-s)
+
C\sup_{\alpha\in\mathcal A_{N,s}}
\mathbb E\big[W_1(\mu_t^{N,s,x,\alpha},\mu_x^N)\big]
\le C\sqrt{t-s}.
\]
Combining the spatial and temporal estimates proves the claim.
\end{proof}

For $\mu \in \mathcal{P}(\mathbb{T}^d)$, write $\mu^{\otimes N}$ for the
$N$-fold product of the probability measure $\mu$. Define
\[
\hat{v}^N(t,\mu)
:=
\int_{\mathbb{T}^{dN}} v^N(t,\mathbf{y}) \, \mu^{\otimes N}(d\mathbf{y}), \qquad (t,\mu) \in [0,T] \times \mathcal{P}(\mathbb{T}^d).
\]
Define
    \begin{align}\label{eq:alpha}
        \alpha(N) :=
         \left\{
    \begin{aligned}
        &CN^{-\frac{1}{2}}, &~\mbox{if}~ d = 1,
        \\
        &CN^{-\frac{1}{2}}\ln(N), &~\mbox{if}~ d = 2,
        \\
        &CN^{-\frac{1}{d}}, &~\mbox{if}~ d > 2.
	    \end{aligned}
	    \right.
	    \end{align}
The next estimates are standard in the Fourier--Wasserstein
doubling-of-variables argument. We record them with references, emphasizing
that the constants are uniform in \(N\).
\begin{lemma}[Fourier Lipschitz estimate]\label{lem:vhat-lipschitz}
Under Assumption~\ref{ass:rate}, there is a constant $C>0$, independent of
$N$, such that
\[
\left|\hat{v}^N(t,\mu)-\hat{v}^N(t,\nu)\right|
\leq C\rho_F(\mu,\nu).
\]
\end{lemma}
\begin{proof}[Proof sketch]
The derivative estimates of Proposition~\ref{prop:particle-regularity} imply a
uniform bound on the Lions derivative of the product lift in the dual Fourier
norm associated with \(\rho_F\). The conclusion follows by integrating this
derivative along the segment joining \(\mu\) and \(\nu\), exactly as in
\cite{BayraktarEkrenZhangRate}.
\end{proof}

\begin{prop}[Product lift and empirical error]\label{prop:vhat-empirical}
For any $x,y \in \mathbb{T}^{dN}$, the value function satisfies
\[
\left|v^N(t,x)-\hat{v}^N(t,\mu_x^N)\right| \leq C\alpha(N),
\]
where $C>0$ is independent of $N$. Together with
Lemma~\ref{lem:vhat-lipschitz}, this immediately implies
\[
\left|v^N(t,x)-v^N(t,y)\right|
\leq C\left(\rho_F(\mu_x^N,\mu_y^N)+\alpha(N)\right).
\]
\end{prop}
\begin{proof}[Proof sketch]
Taking \(Y_1,\ldots,Y_N\) i.i.d. with law \(\mu_x^N\), the modulus estimate in
Proposition~\ref{prop:particle-regularity} reduces the first bound to the
compact-space empirical \(W_1\) estimate of \cite{FournierGuillin}. The second
bound follows by applying the first one at \(x\) and \(y\), and then using
Lemma~\ref{lem:vhat-lipschitz}.
\end{proof}

\begin{lemma}[Deterministic empirical approximation]\label{lem:empirical-approximation}
For any $\mu \in \mathcal{P}(\mathbb{T}^d)$,
\[
\inf_{x \in \mathbb{T}^{dN}} \rho_F(\mu_x^N,\mu) \leq C\alpha(N).
\]
\end{lemma}
\begin{proof}[Proof sketch]
Since \(k^\ast>d/2+2\), the Fourier metric is controlled by \(W_1\). Applying
the compact-space empirical estimate of \cite{FournierGuillin} to i.i.d.
samples from \(\mu\), and then choosing one realization with error no larger
than the expectation, gives the claim.
\end{proof}

\begin{lemma}[Inf-convolution contact point]\label{lem:inf-convolution-transfer}
Assume that \((t_0,\mu_0,z_0)\in [0,T)\times \mathcal P(\mathbb T^d)\times \mathbb T^d\) is a strict local minimum point of
$\overline V_\varepsilon^N-\Phi.$
Then, for \(\varepsilon>0\) sufficiently small, there exists
\((s_0,x_0,w_0)\in [0,T)\times (\mathbb T^d)^N\times \mathbb T^d\) such that
\[
\overline V_\varepsilon^N(t_0,\mu_0,z_0)
=
\overline V^N(s_0,x_0,w_0)
+
\frac{|t_0-s_0|^2+|z_0-w_0|^2+\rho_F^2(\mu_0,\mu_{x_0}^N)}
{2\varepsilon},
\]
and \(\overline V^N-\Phi_\varepsilon\) has a local minimum at \((s_0,x_0,w_0)\).
\end{lemma}
\begin{proof}
The infimum is attained by compactness and continuity. The quadratic time
penalty keeps \(s_0<T\) for \(\varepsilon\) small because \(t_0<T\). The local
minimum property follows by the standard inf-convolution transfer argument:
combine the definition of \(\overline V_\varepsilon^N\) near
\((t_0,\mu_0,z_0)\) with the local minimality of
\(\overline V_\varepsilon^N-\Phi\), then take the supremum over \(z\) in the
definition of \(\Phi_\varepsilon\).
\end{proof}

\begin{lemma}[Convolution derivative identities]\label{lem:derivative_calcu}
Assume that \(\Phi\) is partially \(C^2\)-regular. Then the following
derivative relations hold:
    \begin{align*}
        &\pa_t\Phi_\eps(s_0,x_0,w_0)
        ~ = ~
        \partial_t\Phi(t_0,\mu_0,z_0)
        =
        \frac{t_0-s_0}{\varepsilon},
        \\&
         ND_{x_i}\Phi_\varepsilon(s_0,x_0,w_0)
         =
         D_\mu\Phi(t_0,\mu_0,z_0)(x_{0,i})
        =
        \frac{1}{2\varepsilon}
        D_\mu\!\left[\mu\mapsto \rho_F^2(\mu,\mu_{x_0}^N)\right](\mu_0)(x_{0,i}),
        \\&
        D_{x\mu}\Phi(t_0,\mu_0,z_0)(\xi)
        =
        \frac{1}{2\varepsilon}
        D_xD_\mu\!\left[\mu\mapsto \rho_F^2(\mu,\mu_{x_0}^N)\right](\mu_0)(\xi),
        \\&
        D_{x_i}^2\Phi_\varepsilon(s_0,x_0,w_0)
        =
        \frac{1}{N}
        D_{x\mu}\Phi(t_0,\mu_0,z_0)(x_{0,i})
        -
        \frac{1}{2\varepsilon N^2}
        D_{\mu\mu}^2\!\left[\mu\mapsto \rho_F^2(\mu,\mu_0)\right](\mu_{x_0}^N)(x_{0,i},x_{0,i}),
    \end{align*}
    Moreover, for \(\varepsilon>0\) sufficiently small the maximizer
    \[
    z(w):=\operatorname*{arg\,max}_{z\in\mathbb T^d}
    \left\{
    \Phi(t_0,\mu_0,z)-\frac{1}{2\varepsilon}|w-z|^2
    \right\}
    \]
    is unique and depends \(C^1\)-smoothly on \(w\). Consequently,
    \(w\mapsto\Phi_\varepsilon(s_0,x_0,w)\) is \(C^2\). With
    \(H(w):=D_z^2\Phi(t_0,\mu_0,z(w))\), one has
    \[
    D_{ww}^2\Phi_\varepsilon(s_0,x_0,w)
    =
    H(w)\bigl(I_d-\varepsilon H(w)\bigr)^{-1}
    \ge H(w).
    \]
\end{lemma}
\begin{proof}
By Lemma~\ref{lem:inf-convolution-transfer}, the point
\((t_0,\mu_0,z_0)\) is a local maximum of
\[
\Phi(t,\mu,z)
-
\frac{|t-s_0|^2+|z-w_0|^2+\rho_F^2(\mu,\mu_{x_0}^N)}{2\varepsilon}
\]
up to an additive constant. The first-order optimality conditions in \(t\) and
\(\mu\) give
\[
\partial_t\Phi(t_0,\mu_0,z_0)=\frac{t_0-s_0}{\varepsilon},
\qquad
D_\mu\Phi(t_0,\mu_0,z_0)(\xi)
=
\frac{1}{2\varepsilon}
D_\mu \rho_F^2(\cdot,\mu_{x_0}^N)(\mu_0)(\xi).
\]
Differentiating the second identity in \(\xi\) gives the formula for
\(D_{x\mu}\Phi\).

Since \(z_0\) is an optimizer in the definition of
\(\Phi_\varepsilon(s_0,x_0,w_0)\), the envelope theorem gives the time
identity for \(\Phi_\varepsilon\). For the particle variables, the empirical
chain rule gives
\[
D_{x_i}\rho_F^2(\mu_x^N,\mu_0)
=
\frac1N D_\mu\rho_F^2(\cdot,\mu_0)(\mu_x^N)(x_i).
\]
Using the symmetry of the quadratic Fourier distance and the optimality
condition above, this yields
\[
N D_{x_i}\Phi_\varepsilon(s_0,x_0,w_0)
=
D_\mu\Phi(t_0,\mu_0,z_0)(x_{0,i}).
\]
Differentiating once more in \(x_i\) gives the stated second-derivative
identity, with the additional \(N^{-2}\) term coming from the derivative of the
empirical measure.

For the \(w\)-derivatives, fix \(s_0,x_0\) and write
\[
u(w):=\Phi_\varepsilon(s_0,x_0,w),
\qquad
\varphi(z):=\Phi(t_0,\mu_0,z).
\]
Because \(D_z^2\varphi\) is bounded, for sufficiently small \(\varepsilon\)
the function
\(
z\mapsto \varphi(z)-\frac{|w-z|^2}{2\varepsilon}
\)
is strictly concave on the coordinate neighborhood containing its maximizers.
The quadratic penalty keeps every maximizer in this neighborhood, so \(z(w)\)
is unique. Its first-order condition is
\[
w=z(w)-\varepsilon D_z\varphi(z(w)).
\]
Since \(I_d-\varepsilon D_z^2\varphi\) is positive definite, the implicit
function theorem gives
\[
D_wz(w)=\bigl(I_d-\varepsilon H(w)\bigr)^{-1}.
\]
The envelope identity \(D_wu(w)=D_z\varphi(z(w))\) now yields
\[
D_w^2u(w)=H(w)\bigl(I_d-\varepsilon H(w)\bigr)^{-1}.
\]
Finally,
\[
D_w^2u(w)-H(w)
=
\varepsilon H(w)^2\bigl(I_d-\varepsilon H(w)\bigr)^{-1}
\ge0,
\]
which proves the asserted matrix inequality.
\end{proof}

\begin{lemma}[Penalization localization]\label{lem:Difference_estimate}
    \begin{align*}
        |t_0 - s_0| \le C\eps^{2/3},
        \quad
        |z_0 - w_0| \le C\eps,
        \quad
        \rho_F(\mu_0,\mu^N_{x_0})
        \le C\Big(\eps + \alpha(N) + \sqrt{\eps\alpha(N)}\Big).
    \end{align*}
\end{lemma}
\begin{proof}
These are the standard penalization estimates; see
\cite{BayraktarEkrenZhangRate}. Set
\[
h:=|t_0-s_0|,
\qquad q:=|z_0-w_0|,
\qquad r:=\rho_F(\mu_0,\mu_{x_0}^N).
\]
By minimality, comparison with \((t_0,x_0,w_0)\) gives
\[
\frac{h^2}{2\varepsilon}
\le
\overline V^N(t_0,x_0,w_0)-\overline V^N(s_0,x_0,w_0)
\le C h^{1/2},
\]
where the last inequality is the time modulus from
Proposition~\ref{prop:particle-regularity}. Hence
\(h\le C\varepsilon^{2/3}\). Similarly, comparison with
\((s_0,x_0,z_0)\) and the spatial Lipschitz estimate for common translations
give
\(
\frac{q^2}{2\varepsilon}\le Cq\), and hence \(
q\le C\varepsilon.
\)
For the measure term, choose \(y\in(\mathbb T^d)^N\) such that
\(\rho_F(\mu_y^N,\mu_0)\le C\alpha(N)\), using
Lemma~\ref{lem:empirical-approximation}. Comparing the minimizer with
\((s_0,y,w_0)\) and applying Proposition~\ref{prop:vhat-empirical} gives
\[
\frac{r^2}{2\varepsilon}
\le
C\bigl(\rho_F(\mu_{x_0}^N,\mu_y^N)+\alpha(N)\bigr)
+\frac{C\alpha(N)^2}{\varepsilon}.
\]
Since \(\rho_F(\mu_{x_0}^N,\mu_y^N)\le r+C\alpha(N)\), the quadratic term can
be absorbed into the left-hand side, yielding
\(
r\le C\left(\varepsilon+\alpha(N)+\sqrt{\varepsilon\alpha(N)}\right).
\)
\end{proof}

\begin{prop}[Residual estimate]\label{prop:Key_estimate}
Under Assumption~\ref{ass:rate}, let \(E(N,\varepsilon)\) be the residual defined in
\eqref{eq:error-inf}. If
\(
\varepsilon=\alpha(N),
\) as in \eqref{eq:alpha},
then, for a constant \(C>0\) independent of \(N\),
\[
|E(N,\eps)|
\le
\begin{cases}
C N^{-1/6}, & d=1,\\[1mm]
C N^{-1/6}(\ln N)^{1/3}, & d=2,\\[1mm]
C N^{-1/(3d)}, & d>2.
\end{cases}
\]
\end{prop}
\begin{proof}
Throughout the proof, \(C\) denotes a constant independent of \(N\) and
\(\varepsilon\), which may change from line to line. We write
\[
m_N:=\mu^N_{x_0}=\frac1N\sum_{i=1}^N\delta_{x_{0,i}},
\qquad
r:=\rho_F(\mu_0,m_N),
\qquad
\nu_0:=\mu_0^{z_0},\quad \nu_N:=m_N^{w_0},
\qquad
\Sigma:=\sigma\sigma^\top.
\]
We prove the estimate for the inf-convolution residual
\(E(N,\varepsilon)\). The corresponding subsolution property of
\(\overline{V}^{N,\eps}\), with a residual bounded by \(|E(N,\eps)|\), follows
in the same way. All estimates below are in absolute value.

By Lemma~\ref{lem:derivative_calcu}, the first-order derivatives of the test
function satisfy
\[
\partial_t\Phi_\varepsilon(s_0,x_0,w_0)
=
\partial_t\Phi(t_0,\mu_0,z_0)
=
\frac{t_0-s_0}{\varepsilon},
\]
and, for every \(i=1,\dots,N\),
\[
N D_{x_i}\Phi_\varepsilon(s_0,x_0,w_0)
=
D_\mu\Phi(t_0,\mu_0,z_0)(x_{0,i}).
\]
Moreover,
\begin{align*}
D_\mu\Phi(t_0,\mu_0,z_0)(x)
&=
\frac1{2\varepsilon}
D_\mu\big[\mu\mapsto \rho_F^2(\mu,m_N)\big](\mu_0)(x),\\
D_{x\mu}\Phi(t_0,\mu_0,z_0)(x)
&=
\frac1{2\varepsilon}
D_{x\mu}\big[\mu\mapsto \rho_F^2(\mu,m_N)\big](\mu_0)(x).
\end{align*}
For the second derivatives, Lemma~\ref{lem:derivative_calcu} gives
\[
D^2_{x_i}\Phi_\varepsilon(s_0,x_0,w_0)
=
\frac1N D_{x\mu}\Phi(t_0,\mu_0,z_0)(x_{0,i})
-\frac1{2\varepsilon N^2}
D^2_{\mu\mu}\big[\mu\mapsto \rho_F^2(\mu,\mu_0)\big](m_N)(x_{0,i},x_{0,i}).
\]

The Fourier metric gives the standard bounds
\[
\big\|D_\mu\Phi(t_0,\mu_0,z_0)\big\|_{C^1}
+
\big\|D_{x\mu}\Phi(t_0,\mu_0,z_0)\big\|_{C^0}
\le
C\frac{r}{\varepsilon},
\]
and
\[
\sup_{x,y\in\mathbb T^d}
\left|
D^2_{\mu\mu}
\big[\mu\mapsto \rho_F^2(\mu,\mu_0)\big](m_N)(x,y)
\right|
\le C.
\]
Indeed, if
\[
\rho_F^2(\mu,\nu)
=
\sum_k c_k
\left(
\int_{\mathbb T^d} e_k(x)\,(\mu-\nu)(dx)
\right)^2,
\]
then
\[
D_\mu\rho_F^2(\mu,\nu)(x)
=
2\sum_k c_k
\left(
\int_{\mathbb T^d} e_k(y)\,(\mu-\nu)(dy)
\right)\nabla e_k(x),
\]
\[
D_{x\mu}\rho_F^2(\mu,\nu)(x)
=
2\sum_k c_k
\left(
\int_{\mathbb T^d} e_k(y)\,(\mu-\nu)(dy)
\right)D^2 e_k(x),
\]
and
\[
D^2_{\mu\mu}\rho_F^2(\mu,\nu)(x,y)
=
2\sum_k c_k\,\nabla e_k(x)\otimes \nabla e_k(y).
\]
The preceding estimates follow from the Cauchy--Schwarz inequality and the
definition of \(\rho_F\).

We compare the Hamiltonian parts of the limiting and particle operators in the
residual \(E\). For fixed
\(a\in A\), consider
\[
H_0(a):=\int_{\mathbb T^d}
\left[
\ell(t_0,x+z_0,\nu_0,a)
+
b(t_0,x+z_0,\nu_0,a)\cdot D_\mu\Phi(t_0,\mu_0,z_0)(x)
\right]\mu_0(dx)
\]
and
\[
H_N(a):=\frac1N\sum_{i=1}^N
\left[
\ell(s_0,x_{0,i}+w_0,\nu_N,a)
+
b(s_0,x_{0,i}+w_0,\nu_N,a)
\cdot D_\mu\Phi(t_0,\mu_0,z_0)(x_{0,i})
\right].
\]
The absolute difference between the suprema is bounded by
\(
\left|\sup_{a\in A} H_N(a)-\sup_{a\in A}H_0(a)\right|
\le
\sup_{a\in A}|H_N(a)-H_0(a)|
\).

For fixed \(a\), we split \(H_N(a)-H_0(a)\) into the contribution from
replacing \((t_0,\nu_0,z_0)\) with \((s_0,\nu_N,w_0)\) and the contribution
from replacing \(\mu_0\) with \(m_N\). The first contribution is bounded by
\[
C\left(|t_0-s_0|^{1/2}+|z_0-w_0|+r\right)
\left(
1+\frac{r}{\varepsilon}
\right).
\]
Here we use
\(
\rho_F(\nu_0,\nu_N)
\le C\bigl(\rho_F(\mu_0,m_N)+|z_0-w_0|\bigr)
\),
the regularity of \(\ell\) and \(b\) in time, state, and measure,
together with
\(
\big\|D_\mu\Phi(t_0,\mu_0,z_0)\big\|_\infty
\le C r/\varepsilon
\).
For the second part, the duality estimate associated with \(\rho_F\) gives
\[
\left|
\int_{\mathbb T^d} g(x)\,(m_N-\mu_0)(dx)
\right|
\le C\|g\|_F\,\rho_F(m_N,\mu_0),
\]
for smooth \(g\). Applying this estimate to the \(\ell\)-term gives a contribution
bounded by \(Cr\). Applying it to the \(b\cdot D_\mu\Phi\)-term gives a
contribution bounded by
\[
C r
\big\|b(s_0,\cdot+w_0,\nu_N,a)\cdot D_\mu\Phi(t_0,\mu_0,z_0)\big\|_F
\le
C\frac{r^2}{\varepsilon}.
\]
Therefore the Hamiltonian part satisfies
\[
\left|\sup_{a\in A} H_N(a)-\sup_{a\in A}H_0(a)\right|
\le
C\left[
\left(|t_0-s_0|^{1/2}+|z_0-w_0|+r\right)
\left(1+\frac r\varepsilon\right)
+
\frac{r^2}{\varepsilon}
\right].
\]

It remains to compare the idiosyncratic diffusion terms. Since
\(\Sigma=\sigma\sigma^\top\) is constant, the second-derivative identity gives
\[
\sum_{i=1}^N\operatorname{Tr}\bigl(
\Sigma D^2_{x_i}\Phi_\varepsilon(s_0,x_0,w_0)\bigr)
=
\frac1N\sum_{i=1}^N\operatorname{Tr}\bigl(
\Sigma D_{x\mu}\Phi(t_0,\mu_0,z_0)(x_{0,i})\bigr)+R_N,
\]
where \(|R_N|\le C/(N\varepsilon)\). Applying Fourier duality to
\(g(x):=\operatorname{Tr}(\Sigma D_{x\mu}\Phi(t_0,\mu_0,z_0)(x))\) gives
\[
\left|
\int_{\mathbb T^d}g(x)\,\mu_0(dx)-\frac1N\sum_{i=1}^Ng(x_{0,i})
\right|
\le Cr\|g\|_F
\le C\frac{r^2}{\varepsilon}.
\]
Hence the diffusion contribution is bounded by
\(C r^2/\varepsilon+C/(N\varepsilon)\).

Combining the Hamiltonian and diffusion estimates gives
\[
\left|E(N,\varepsilon)\right|
\le
C\left[
\left(|t_0-s_0|^{1/2}+|z_0-w_0|+r\right)
\left(1+\frac r\varepsilon\right)
+
\frac{r^2}{\varepsilon}
+
\frac1{N\varepsilon}
\right].
\]
By Lemma~\ref{lem:Difference_estimate},
\(
|t_0-s_0|\le C\varepsilon^{2/3},
\qquad
|z_0-w_0|\le C\varepsilon
\).
Since \(r=\rho_F(\mu_0,m_N)\), the same lemma also gives
\(
r
\le
C\left(\varepsilon+\alpha(N)+\sqrt{\varepsilon\alpha(N)}\right)
\).
We choose
\(
\varepsilon=\alpha(N).
\)
Then
\(
r\le C\varepsilon,
\,
|z_0-w_0|\le C\varepsilon,
\,
|t_0-s_0|^{1/2}\le C\varepsilon^{1/3}
\).
Thus
\[
|t_0-s_0|^{1/2}+|z_0-w_0|+r
\le C\varepsilon^{1/3},
\qquad
\frac r\varepsilon\le C,
\qquad
\frac{r^2}{\varepsilon}\le C\varepsilon.
\]
Substituting these bounds into the previous estimate gives
\[
\left|E(N,\varepsilon)\right|
\le
C\left(
\varepsilon^{1/3}
+
\varepsilon
+
\frac1{N\varepsilon}
\right).
\]
Since \(\varepsilon=\alpha(N)\), the last term is also bounded by
\(C\alpha(N)^{1/3}\). Indeed, if \(d=1\), then
\(
\alpha(N)=N^{-1/2},
\,
\frac1{N\alpha(N)}=N^{-1/2}\le N^{-1/6}=\alpha(N)^{1/3}
\).
If \(d=2\), then \(\alpha(N)=N^{-1/2}\ln N\), and therefore
\[
\frac1{N\alpha(N)}
=
N^{-1/2}(\ln N)^{-1}
\le
C N^{-1/6}(\ln N)^{1/3}
=
C\alpha(N)^{1/3}.
\]
Finally, if \(d>2\), then
\(
\alpha(N)=N^{-1/d},
\,
\frac1{N\alpha(N)}
=
N^{-1+1/d}
\le
N^{-1/(3d)}
=
\alpha(N)^{1/3}.
\)
Hence
\(
|E(N,\varepsilon)|
\le
C\alpha(N)^{1/3}
\).
The same calculation, with the differences reversed and with the
second-derivative identity for the transferred inf-convolution test function,
gives the same bound. This concludes the proof.
\end{proof}

\begin{proof}[Proof of Theorem~\ref{thm:rate}]
We first recall how the preliminary results enter. Since
Assumption~\ref{ass:rate} implies Assumption~\ref{ass:comparison},
Proposition~\ref{prop:comparison-limit} gives comparison for the limiting
equation. The dynamic programming principle and
Proposition~\ref{prop:particle-regularity} identify \(V^N\) with the classical
solution of \eqref{eq:N_particle} and provide the derivative and modulus
estimates used below. Lemma~\ref{lem:inf-convolution-transfer} transfers a
touching point of the inf-convolution to the particle equation, and
Lemma~\ref{lem:derivative_calcu} computes the corresponding derivatives.
Lemma~\ref{lem:vhat-lipschitz}, Proposition~\ref{prop:vhat-empirical}, and
Lemma~\ref{lem:empirical-approximation} enter through
Lemma~\ref{lem:Difference_estimate}, which localizes the minimizing triple.
These ingredients are assembled in Proposition~\ref{prop:Key_estimate}, the
residual estimate that drives the comparison argument.

Set
\(
\delta_N:=\alpha(N)^{1/3}.
\)
Take \(\varepsilon=\alpha(N)\), up to a harmless fixed multiplicative
constant, and
write
\[
\overline V(t,\mu,z)=V(t,(I_d+z)_\#\mu),
\qquad
\overline V^N(t,x,z)=V^N(t,x+z).
\]
Define the inf- and sup-convolutions
\[
\overline V^N_\varepsilon(t,\mu,z)
:=
\inf_{s,y,w}
\left\{
\overline V^N(s,y,w)
+
\frac{|t-s|^2+|z-w|^2+\rho_F^2(\mu,\mu_y^N)}{2\varepsilon}
\right\},
\]
and
\[
\overline V^{N,\varepsilon}(t,\mu,z)
:=
\sup_{s,y,w}
\left\{
\overline V^N(s,y,w)
-
\frac{|t-s|^2+|z-w|^2+\rho_F^2(\mu,\mu_y^N)}{2\varepsilon}
\right\}.
\]
Here \((s,y,w)\) ranges over
\([0,T]\times(\mathbb T^d)^N\times\mathbb T^d\). We prove only that
\(\overline V^N_\varepsilon\) is a viscosity supersolution up to
\(E(N,\eps)\); the subsolution argument for
\(\overline V^{N,\varepsilon}\) is analogous. The key point is that
Proposition~\ref{prop:Key_estimate} turns these regularizations into
approximate solutions of the translated limiting equation. Indeed, if a smooth
test function touches \(\overline V^N_\varepsilon\) from below, then
Lemma~\ref{lem:inf-convolution-transfer} gives a particle contact point and
Lemma~\ref{lem:derivative_calcu} identifies the derivatives of the transferred
test function at that point. Applying the transformed particle equation produces the residual
\(E(N,\varepsilon)\) in \eqref{eq:error-inf};
Proposition~\ref{prop:Key_estimate} bounds it by \(C\delta_N\), giving the
supersolution inequality for \(\overline V^N_\varepsilon\). 

Consequently, for \(K_N=C\delta_N\) with \(C\) large enough,
\[
U^+(t,\mu,z)
:=
\overline V^N_\varepsilon(t,\mu,z)+K_N(T-t)+K_N
\]
is a viscosity supersolution, while
\[
U^-(t,\mu,z)
:=
\overline V^{N,\varepsilon}(t,\mu,z)-K_N(T-t)-K_N
\]
is a viscosity subsolution. The terminal estimates use the same localization.
For example, for any \((s,y,w)\),
\[
\overline V^N(s,y,w)
\ge
G((I_d+w)_\#\mu_y^N)-C\sqrt{T-s},
\]
by the time modulus in Proposition~\ref{prop:particle-regularity}. The
smoothness of \(G\) and the translation stability of \(\rho_F\) give
\[
G((I_d+w)_\#\mu_y^N)
\ge
G((I_d+z)_\#\mu)
-
C\bigl(|z-w|+\rho_F(\mu,\mu_y^N)\bigr).
\]
Optimizing these bounds against the quadratic penalty gives
\[
\overline V^N_\varepsilon(T,\mu,z)
\ge
G((I_d+z)_\#\mu)-C\varepsilon^{1/3}.
\]
The analogous argument for the sup-convolution gives
\[
\overline V^{N,\varepsilon}(T,\mu,z)
\le
G((I_d+z)_\#\mu)+C\varepsilon^{1/3}.
\]
Since \(\varepsilon=\alpha(N)\) and
\(\delta_N=\alpha(N)^{1/3}\), increasing \(K_N\) if
necessary makes \(U^+(T,\mu,z)\ge G((I_d+z)_\#\mu)\) and
\(U^-(T,\mu,z)\le G((I_d+z)_\#\mu)\). Comparison for the limiting equation
then gives, for all \((t,x)\),
\[
V(t,\mu_x^N)\le V^N(t,x)+C\delta_N,
\qquad
V^N(t,x)\le V(t,\mu_x^N)+C\delta_N.
\]
Indeed, in both convolutions one uses the admissible competitor
\((s,y,w)=(t,x,0)\) after setting \(\mu=\mu_x^N\) and \(z=0\). Hence
\[
\sup_{(t,x)\in[0,T]\times(\mathbb T^d)^N}
\left|V^N(t,x)-V(t,\mu_x^N)\right|
\le C\delta_N.
\]
Finally, expanding \(\delta_N=\alpha(N)^{1/3}\) gives the desired estimate.
\end{proof}

{\footnotesize
\setlength{\bibsep}{1pt}
\bibliographystyle{plain}
\bibliography{references}
}

\end{document}